\documentclass[reqno,12pt,oneside]{amsart}
\usepackage{amsmath,amssymb,amsthm,mathrsfs}
\usepackage{epsfig,color}
\usepackage[latin1]{inputenc}
\usepackage[english]{babel}
\usepackage{mathtools}
\usepackage{braket}
\usepackage[toc]{appendix}
\usepackage{relsize}
\usepackage{enumitem}
\usepackage[final]{showkeys}
\usepackage{csquotes}
\usepackage{verbatim}
\usepackage{cite}

\usepackage[margin=1in]{geometry}

\usepackage{esint}

\usepackage{hyperref}
\hypersetup{
	colorlinks=true,
	linkcolor=blue,
	citecolor=blue,
	urlcolor=black,
	linktoc=all
}

\allowdisplaybreaks

\newcommand{\rn}{\mathbb{R}^n}
\newcommand{\R}{\mathbb{R}}
\newcommand{\N}{\mathbb{N}}
\newcommand{\W}{\mathbb{W}}
\newcommand{\<}{\langle}
\renewcommand{\>}{\rangle}
\renewcommand{\a}{\alpha}
\renewcommand{\b}{\beta}

\renewcommand{\l}{\lambda}
\renewcommand{\L}{\Lambda}
\newcommand{\vphi}{\varphi}
\newcommand{\ov}{\overline}
\newcommand{\Om}{\Omega}

\newcommand{\Nu}{\widetilde{\mathcal{V}}}
\newcommand{\dive}{\mbox{div}}
\newcommand{\PV}{\mbox{P.V.}}
\newcommand{\Haus}{\mathcal{H}}
\newcommand{\dist}{\mbox{dist}}
\newcommand{\ccp}{\texttt{ccp}}
\renewcommand{\S}{\mathcal{S}}
\newcommand{\T}{\mathcal{T}}
\newcommand{\loc}{{\rm loc}}
\newcommand{\Dini}{{\rm Dini}}
\newcommand{\coloneqq}{\mathrel{\mathop:}=}
\def\sm{\setminus}
\def\d{\delta}
\def\g{q}
\def\G{\Gamma}
\def\tg{\tilde{g}}
\def\tf{\tilde{f}}
\def\tw{\widetilde{w}}
\def\tu{\tilde{u}}

\def\tx{\tilde{x}}
\def\vrho{\varrho}
\def\th{\tilde{h}}

\DeclareMathOperator*{\osc}{osc}
\DeclareMathOperator*{\supp}{supp}

\def\Xint#1{\mathchoice
	{\XXint\displaystyle\textstyle{#1}}%
	{\XXint\textstyle\scriptstyle{#1}}%
	{\XXint\scriptstyle\scriptscriptstyle{#1}}%
	{\XXint\scriptscriptstyle\scriptscriptstyle{#1}}%
	\!\int}
\def\XXint#1#2#3{{\setbox0=\hbox{$#1{#2#3}{\int}$ }
		\vcenter{\hbox{$#2#3$ }}\kern-.6\wd0}}

\def\dashint{\Xint-}

\newtheorem*{theorem*}{Theorem}
\newtheorem{theorem}{Theorem}[section]
\newtheorem{lemma}[theorem]{Lemma}
\newtheorem{proposition}[theorem]{Proposition}

\newtheorem{corollary}[theorem]{Corollary}

\numberwithin{equation}{section}

\title[Global regularity and Hopf lemma for mixed local-nonlocal operators]{Global gradient regularity and a Hopf lemma for quasilinear operators of mixed local-nonlocal type}

\author{Carlo Alberto Antonini}
\author{Matteo Cozzi}

\address{
\vspace{20pt}
\vspace{-\baselineskip}
\newline
\textit{Carlo Alberto Antonini}
\newline
Universit\`a degli Studi di Milano, Dipartimento di Matematica ``Federigo Enriques'', Via Saldini 50, 20133 Milan, Italy
\newline
\textit{E-mail address}: \textit{\tt carlo.antonini@unimi.it}
}

\address{
\vspace{-\baselineskip}
\newline
\textit{Matteo Cozzi}
\newline
Universit\`a degli Studi di Milano, Dipartimento di Matematica ``Federigo Enriques'', Via Saldini 50, 20133 Milan, Italy
\newline
\textit{E-mail address}: \textit{\tt matteo.cozzi@unimi.it}
}

\begin{document}

\subjclass[2020]{35B65, 35J60, 35M12, 35R11}
\keywords{Boundary regularity, mixed local and nonlocal operators,~$p$-Laplacian, fractional~$(s,q)$-Laplacian, Hopf lemma}

\maketitle

\begin{abstract}
We address some regularity issues for mixed local-nonlocal quasilinear operators modeled upon the sum of a~$p$-Laplacian and of a fractional~$(s, q)$-Laplacian. Under suitable assumptions on the right-hand sides and the outer data, we show that weak solutions of the Dirichlet problem are~$C^{1, \theta}$-regular up to the boundary. In addition, we establish a Hopf type lemma for positive supersolutions. Both results hold assuming the boundary of the reference domain to be merely of class~$C^{1, \alpha}$, while for the regularity result we also require that~$p > s q$.
\end{abstract}


\vspace{40pt}

\section{Introduction}\label{sec:intro}

\noindent
Linear second-order elliptic integrodifferential operators are important analytical objects connected with diffusion processes, as, by the L\'evy-Khintchine formula, they are the infinitesimal generator of a general L\'evy process. They are structured as the sum of two parts: a local second-order elliptic operator associated with a Brownian motion and a purely nonlocal integrodifferential operator of fractional order modeling a L\'evy flight.

In the last twenty years a great deal of research has focused on the study of purely nonlocal operators. Comparatively, much fewer articles have investigated the effect of coupling local and nonlocal terms. Without any claim to exhaustiveness, we point out the papers~\cite{CKSV10, CKSV12} by Chen, Kim, Song \& Vondra\v{c}ek, which established Green function estimates and a boundary Harnack inequality for the Dirichlet problem, and the series of contributions~\cite{BDVV21, BDVV22, BDVV23b, BDVV23c} by Biagi, Dipierro, Valdinoci \& Vecchi, where a number of properties enjoyed by the solutions of linear and semilinear equations are studied. We note in passing that the scope of these papers was confined to the model operator determined by the superposition of the Laplacian and of one of its fractional powers.

More recently, attention has been given to quasilinear generalizations of these models, such as the sum of a~$p$-Laplacian and of a fractional~$(s, q)$-Laplacian---see, for instance,~\cite{BDVV23,BMV,dS20,GK22,GL,min22}. The article~\cite{min22} by De Filippis \& Mingione, in particular, obtained several regularity results for a general class of equations for this kind of operators, such as the interior H\"older continuity of the gradient of the solutions and their global almost Lipschitz character, under the assumption that~$p > s q$. In the present paper, we take inspiration from~\cite{min22} and focus on two connected issues: the boundary gradient regularity of solutions and the validity of a Hopf type lemma. Our results apply to a large family of quasilinear elliptic operators of mixed local-nonlocal type, which we now proceed to define.

\vspace{4pt}

Let~$n \ge 2$ be an integer,~$p, q \in (1, +\infty)$, and~$s\in (0,1)$. Let~$\Omega \subset \R^n$ be a bounded open set. We consider the operator
\begin{equation}\label{defQ}
Qu \coloneqq Q_L \, u + Q_N \, u,    
\end{equation}
defined as the sum of the local term
$$
Q_L \, u(x) = Q_L^A  \, u(x) \coloneqq - \dive A \big( {x, D u(x)} \big)
$$
and of the nonlocal one
$$
Q_N \, u(x) = Q_N^{\phi, B, s, q} \, u(x) \coloneqq 2 \, \PV \int_{\R^n} \phi \big( {u(x) - u(y)} \big) \frac{B(x, y)}{|x - y|^{n + s q}} \, dy.
$$
Here,~$A: \Omega \times \R^n \to \R^n$ is a continuous vector field such that~$A(x, \cdot) \in C^1(\R^n \setminus \{ 0 \}; \R^n)$ for all~$x \in \Omega$, $A(\cdot, z) \in C^\alpha(\Omega; \R^n)$ for all~$z \in \R^n$, and which satisfies the~$p$-growth and coercivity conditions
\begin{equation} \label{ass:A}
\begin{cases}
|A(x, z)| + |z| |\partial_z A(x, z)| \le \Lambda \left( |z|^2 + \mu^2 \right)^{\frac{p - 2}{2}} |z| & \quad \mbox{for } x \in \Omega, \, z \in \R^n \setminus \{ 0 \}, \\
\left| A(x, z) - A(y, z) \right| \le \Lambda \left( |z|^2 + \mu^2 \right)^{\frac{p - 1}{2}} |x - y|^\alpha & \quad \mbox{for } x, y \in \Omega, \, z \in \R^n, \\
\langle \partial_z A(x, z) \xi, \xi \rangle \ge \Lambda^{-1} \left( |z|^2 + \mu^2 \right)^{\frac{p - 2}{2}} |\xi|^2 & \quad \mbox{for } x \in \Omega, \, z \in \R^n \setminus \{ 0 \}, \, \xi \in \R^n,
\end{cases}
\end{equation}
for some constants~$\alpha \in (0, 1)$,~$\mu \in [0, 1]$, and~$\Lambda \ge 1$, while~$B: \R^n \times \R^n \to [0, +\infty)$ is a measurable function satisfying
\begin{equation} \label{ass:B}
B(x, y) = B(y, x) \quad \mbox{and} \quad \Lambda^{-1} \le B(x, y) \le \Lambda \quad \mbox{for a.a.~} x, y \in \R^n,
\end{equation}
and~$\phi \in C^0(\R)$ is an odd, non-decreasing function fulfilling the~$q$-growth and coercivity assumption
\begin{equation} \label{ass:phi}
\Lambda^{-1} |t|^q \le \phi(t) \, t \le \Lambda |t|^q \quad \mbox{for all } t \in \R.
\end{equation}
The prototypical example of such an operator is, of course,
\begin{equation} \label{proto}
- \Delta_p u + (-\Delta_q)^s u,
\end{equation}
which is obtained by taking~$A(x, z) = |z|^{p - 2} z$,~$\phi(t) = |t|^{q - 2} t$, and~$B$ equal to a constant.

\vspace{4pt}

Our first result concerns the global differentiability of weak solutions to the Dirichlet problem for the operator~$Q$. The notion of weak solution and the relevant functional spaces will be made precise in Section~\ref{sec:prel}.

\begin{theorem}[Global $C^{1,\theta}$-regularity]\label{main:thm}
Let~$p, q \in (1, +\infty)$ and~$s \in (0, 1)$ be such that
\begin{equation}\label{psq}
    p > s q.
\end{equation}
Let~$\Omega \subset \subset \Omega' \subset \R^n$ be bounded open sets, with~$\partial \Omega$ of class~$C^{1, \alpha}$ for some~$\alpha \in (0, 1)$. Suppose that~$A$,~$B$, and~$\phi$ satisfy assumptions~\eqref{ass:A},~\eqref{ass:B}, and~\eqref{ass:phi}. Let~$f \in L^d(\Omega)$ for some~$d > n$ and~$g \in \W^{s, q}(\Omega) \cap W^{1,\infty}(\Omega') \cap C^{1, \alpha}(\partial\Omega)$. Let~$u\in W^{1, p}_g(\Omega) \cap \W_g^{s, q}(\Omega)$ be the weak solution of the Dirichlet problem
\begin{equation} \label{dirprobforu}
\begin{cases}
Q u = f & \quad \mbox{in } \Omega, \\
u = g & \quad \mbox{in } \R^n \setminus \Omega.
\end{cases}
\end{equation}
Then,~$u \in C^{1, \theta}(\overline{\Omega})$ and
$$
\| u \|_{C^{1,\theta}(\ov{\Om})}\leq C,
$$
for some constants~$\theta\in (0,1)$ and~$C>0$ depending only on~$n$,~$p$,~$q$,~$s$,~$\L$,~$d$,~$\alpha$,~$\Omega$, and~$\Omega'$, as well as on~$\| f \|_{L^d(\Om)}$,~$\| g \|_{\W^{s, q}(\Omega)}$,~$\| g \|_{W^{1,\infty}(\Omega')}$, and~$\| g \|_{C^{1, \alpha}(\partial\Omega)}$.
\end{theorem}

Under virtually the same assumptions on~$Q$, interior~$C^{1, \theta}$ estimates and boundary almost Lipschitz regularity were established in~\cite{min22}. Theorem~\ref{main:thm} provides a strengthening of these results, in the case of a sufficiently regular outside datum~$g$. We also point out that, for~$Q$ as in~\eqref{proto} with~$p = q = 2$,~$f \in L^\infty(\Omega)$, and~$g \equiv 0$, global~$C^{1, \theta}$-estimates have been obtained in~\cite{BDVV23b,su}.

Like the majority of the results in~\cite{min22}, Theorem~\ref{main:thm} relies crucially on assumption~\eqref{psq}. This requirement ensures that the local operator~$Q_L$ is the leading term in~\eqref{defQ}, making it increasingly prevailing over~$Q_N$ at smaller scales and ultimately becoming the source of regularity. Clearly,~\eqref{psq} is satisfied if~$p = q$, as, for instance, when~$Q u = - \Delta_p u + (-\Delta_p)^s u$. If~$p < s q$, then the leading term becomes~$Q_N$, from which one should not be able to extract more than the global H\"older continuity of solutions---see~\cite{rosoton} and~\cite{ims}. Different is the case of interior regularity, where, in some cases,~$C^{1, \theta}$ estimates are expected. However, to obtain them, one would need to fully understand the regularizing features of~$Q_N$, something which at the moment is still lacking---see~\cite{BLS18} for some of the most relevant results in this direction.

The proof of Theorem~\ref{main:thm} is reminiscent of that of~\cite[Theorem 5]{min22} for the interior H\"older continuity of the gradient and proceeds as follows. After carrying out a suitable flattening of the boundary, we establish Caccioppoli type estimates for the solution~$u$ of problem~\eqref{dirprobforu} near said flat parts of the boundary. Then, in the same spirit of~\cite{gg,lieb88}, we make use of a perturbation argument and compare~$u$ to the solution of a local, autonomous, homogeneous problem in the half-ball, whose gradient regularity is well understood. This allows us to obtain finer estimates on the gradient of~$u$ and, in conjunction with the Caccioppoli estimates of the previous step, to ultimately get boundary Campanato type estimates for~$Du$. Its H\"older regularity is then recovered as a consequence of the Campanato isomorphism.

\vspace{4pt}

Theorem~\ref{main:thm} gives the~$C^{1, \theta}$-regularity of the solution~$u$ of problem~\eqref{dirprobforu} up to the boundary of~$\Omega$, from the interior. However, no matter how nice the outer datum~$g$ is,~$u$ will in general be no more than Lipschitz across the boundary. This can be deduced as a particular consequence of the second result of the present paper, a Hopf type boundary point lemma for the operator~$Q$.

In order to state and prove this result, we need to impose some additional regularity hypotheses on the operators~$Q_L$ and~$Q_N$. Namely, we require that~$A(\cdot, z) \in C^1(\Omega; \R^n)$ for all~$z \in \R^n$ and that
\begin{equation} \label{ass:A2}
|\partial_x A(x, z)| \le \Lambda \left( |z|^2 + \mu^2 \right)^{\frac{p - 2}{2}} |z| \quad \mbox{for all } x \in \Omega, \, z \in \R^n.
\end{equation}
Note that this is a strengthening of the second line in~\eqref{ass:A}. Concerning the operator~$Q_N$, we assume that~$B \in C^{0, 1}(\R^n \times \R^n)$, with
\begin{equation} \label{ass:B2}
\big| {B(x + w, y + z) - B(x, y)} \big| \le \Lambda \big( {|w| + |z|} \big) \quad \mbox{for all~} x, y, w, z \in \R^n,
\end{equation}
and that~$\phi \in C^1(\R \setminus \{ 0 \})$, with
\begin{equation} \label{ass:phi2}
\Lambda^{-1} |t|^{q - 2} \le \phi'(t) \le \Lambda |t|^{q - 2} \quad \mbox{for all } t \in \R \setminus \{ 0 \}.
\end{equation}
Observe that condition~\eqref{ass:phi2} is stronger than~\eqref{ass:phi} (up to taking a different~$\Lambda$), as~$\phi(0) = 0$---recall that~$\phi$ is an odd continuous function.

Having made these additional assumptions, we can now state our Hopf lemma for~$Q$-su\-per\-har\-mo\-nic functions---as before, see Section~\ref{sec:prel} for definitions. Here and in what follows,~$\nu$ denotes the unit normal vector field of~$\partial \Omega$, pointing outwards from~$\Omega$.

\begin{theorem} [Hopf lemma]\label{Hopflemma}
Let~$p, q \in (1, +\infty)$ and~$s \in (0, 1)$. Let~$\Omega \subset \R^n$ be a bounded open set with boundary of class~$C^{1, \alpha}$, for some~$\alpha \in (0, 1)$. Suppose that~$A$,~$B$, and~$\phi$ satisfy assumptions~\eqref{ass:A},~\eqref{ass:B},~\eqref{ass:A2},~\eqref{ass:B2}, and~\eqref{ass:phi2}. Let~$u\in W^{1, p}(\Omega) \cap \W^{s, q}(\Omega) \cap C^0(\overline{\Omega})$ be a non-negative weak supersolution of~$Q u = 0$ in~$\Omega$, positive in~$\Omega$ and vanishing at a point~$x_0 \in \partial \Omega$. Then,
\begin{equation} \label{normderneg}
\liminf_{h \searrow 0} \frac{u(x_0 - h \nu(x_0))}{h} > 0.
\end{equation}
\end{theorem}

We remark that Theorem~\ref{Hopflemma} holds for every~$p, q \in (1, +\infty)$ and~$s \in (0, 1)$---in particular, assumption~\eqref{psq} is not required here. Indeed, the result is not of perturbative nature and its proof treats both operators~$Q_L$ and~$Q_N$ as equals. In consequence of Theorem~\ref{main:thm}, the linear growth from the boundary implied by~\eqref{normderneg} is optimal when~$p > s q$. We believe it is an interesting question to determine whether a stronger condition might hold when~$p < s q$, such as
$$
\liminf_{h \searrow 0} \frac{u(x_0 - h \nu(x_0))}{h^s} > 0,
$$
in agreement with the Hopf lemmas available for fractional Laplacians---see~\cite{GS16,dpq}.

In its general spirit, the proof of Theorem~\ref{Hopflemma} proceeds similarly to those usually employed to establish Hopf lemmas, via the construction of a suitable positive subsolution. Once this barrier is built, the conclusion then follows from the weak comparison principle---see, e.g., the forthcoming Proposition~\ref{WCPprop}.

A first difficulty that presents to us when building such a barrier comes from the mild regularity assumptions made on the boundary of~$\Omega$, which is only required to be~$C^{1, \alpha}$---in particular, it might not satisfy the interior ball condition. After flattening the boundary through a specific diffeomorphism, this low regularity translates into a transformed operator having coefficients which may blow up near~$x_0$. To overcome this difficulty, we construct an explicit subsolution~$v$ having second derivatives which blow up at a faster rate, with the correct sign. This method is, to the best of our knowledge, rather unexplored even in the case of a single local operator---see~\cite{G33,FG57} for similar approaches. We believe it might be further generalized past the H\"older continuity class and could lead to results for~$C^{1, \Dini}$-regular boundaries, the optimal regularity under which the Hopf lemma holds in the local case---see, e.g.,~\cite{W67,L85,AN16}.

A second difficulty naturally lies in the fact that~$Q$ is the sum of two operators having different scaling and homogeneity properties. In order to circumvent this issue, we actually construct~$v$ in a way that makes it subharmonic for both~$Q_L$ and~$Q_N$ at the same time. As a technical remark, we point out that, to prove that~$v$ is a subsolution of~$Q_N v = 0$ in a neighborhood of~$x_0$, we need both a careful asymptotic analysis of the behavior of the part of~$Q_N v$ localized around~$x_0$ (in the mildly nonlocal regime~$(1- s) q < 1$) and purely nonlocal techniques, adding a large bump function supported away from the boundary as in~\cite{dpq} (in the strongly nonlocal regime~$(1 - s) q \ge 1$).


\vspace{4pt}

Clearly, supersolutions of~$Q u = 0$ might not be differentiable and thus the~$\liminf$ in~\eqref{normderneg} might not in general be a limit. Of course, this is true unless the supersolution~$u$ is a priori assumed to be of class~$C^1(\overline{\Omega})$ or if~$u$ is an actual solution of the equation and~\eqref{psq} is in force, thanks to Theorem~\ref{main:thm}.

In the following result, we showcase this last possibility and provide a unified statement which can be easily proved by combining Theorems~\ref{main:thm} and~\ref{Hopflemma} with the weak and the strong maximum principles for~$Q$---see, e.g., the forthcoming Propositions~\ref{WCPprop} and~\ref{SMPprop}.

\begin{corollary}\label{coroll}
Let~$p,q\in (1,+\infty)$ and~$s\in (0,1)$ be such that~\eqref{psq} holds true. Let~$\Omega \subset \R^n$ be a bounded open set with boundary of class~$C^{1, \alpha}$, for some~$\alpha \in (0, 1)$. Suppose that~$A$,~$B$, and~$\phi$ satisfy assumptions~\eqref{ass:A},~\eqref{ass:B},~\eqref{ass:A2},~\eqref{ass:B2}, and~\eqref{ass:phi2}. Let~$f \in L^d(\Omega)$, for some~$d > n$, be a non-negative function and~$u\in W^{1,p}(\Omega) \cap \W^{s,q}(\Omega)$ be the weak solution of
$$
\begin{cases}
Qu = f & \quad \text{in }\Omega, \\
u = 0 & \quad \text{in } \R^n\setminus \Omega.
\end{cases}
$$
Then,~$u\in C^{1,\theta}(\overline{\Omega})$ for some~$\theta\in (0,1)$ depending only on~$n$,~$p$,~$q$,~$s$,~$\Lambda$,~$d$,~$\alpha$,~$\Omega$, and~$\| f \|_{L^d(\Omega)}$. Furthermore, either~$u\equiv 0$ in~$\R^n$, or~$u>0$ in~$\Omega$ and
$$
-\frac{\partial u}{\partial \nu^{-}}(x_0) \coloneqq \lim_{h \searrow 0}\frac{u\big(x_0-h\,\nu(x_0)\big)}{h}>0,
$$
for every~$x_0\in\partial \Omega$.
\end{corollary}

\vspace{4pt}

The rest of the paper is organized as follows.

In Section~\ref{sec:prel} we state a few definitions and explain some non-standard notation used throughout the paper. Section~\ref{sec:thm} is devoted to the proof of our regularity result, Theorem~\ref{main:thm}. In Section~\ref{sec:weak} we establish a weak comparison principle for the operator~$Q$, which will be used in Section~\ref{hopfsec} to prove Theorem~\ref{Hopflemma}. Section~\ref{sec:smp} contains a strong maximum principle for~$Q$ which, as customary, follows almost immediately from the Hopf boundary point principle of Theorem~\ref{Hopflemma}.

The paper is closed by a couple of technical appendices: Appendix~\ref{app:Linfty} is dedicated to the proof of an~$L^\infty$ estimate used within the proof of Theorem~\ref{main:thm}, while Appendix~\ref{app:Hextproof} contains the computations needed for an extension lemma in the proof of Theorem~\ref{Hopflemma}.

\section{Notation and definitions} \label{sec:prel}

\noindent
In this preliminary section, we collect a few basic definitions and fix some of the terminology that we will use in the rest of the paper. We assume that~$p, q \in (1, +\infty)$,~$s \in (0, 1)$, and that~$\Omega \subset \R^n$ is a bounded open set with Lipschitz boundary.

\begin{itemize}[leftmargin=*]
\itemsep3pt
\item We denote by~$W^{1, p}(\Omega)$ the usual Sobolev space of~$L^p(\Omega)$ weakly differentiable functions having weak gradients in~$L^p(\Omega)$. We endow~$W^{1, p}(\Omega)$ with the usual norm
$$
\| u \|_{W^{1, p}(\Omega)} \coloneqq \| u \|_{L^p(\Omega)} + \| \nabla u \|_{L^p(\Omega)}.
$$
Given~$g \in W^{1, p}(\Omega)$, we indicate with~$W^{1, p}_g(\Omega)$ the subset of~$W^{1, p}(\Omega)$ made up by those function whose traces on~$\partial \Omega$ coincide with that of~$g$. Writing
\begin{align*}
\mathscr{C}_\Omega \coloneqq \, & \big( {\R^n \times \R^n} \big) \setminus \big( {\left( \R^n \setminus \Omega \right) \times \left( \R^n \setminus \Omega \right)} \big) \\
= \, & \big( {\Omega \times \Omega} \big) \cup \big( {\Omega \times \left( \R^n \setminus \Omega \right)} \big) \cup \big( {\left( \R^n \setminus \Omega \right) \times \Omega} \big),
\end{align*}
we define~$\W^{s, q}(\Omega)$ to be the set of measurable functions~$u: \R^n \to \R$ such that~$u|_{\Omega} \in L^q(\Omega)$ and the map~$(x, y) \mapsto |x - y|^{- n -  s q} |u(x) - u(y)|^q$ is integrable over~$\mathscr{C}_\Omega$. We norm this space by
$$
\| u \|_{\W^{s, q}(\Omega)} \coloneqq \| u \|_{L^q(\Omega)} + \left( \iint_{\mathscr{C}_\Omega} \frac{|u(x) - u(y)|^q}{|x - y|^{n + s q}} \, dx dy \right)^{\! \frac{1}{q}}.
$$
Also, given~$g \in \W^{s, q}(\Omega)$, we denote by~$\W_g^{s, q}(\Omega)$ the space composed by all functions in~$\W^{s, q}(\Omega)$ which agree with~$g$ outside of~$\Omega$.

\item Let~$Q$ be as in the introduction. Given~$g \in W^{1, p}(\Omega) \cap \W^{s, q}(\Omega)$ and~$f \in L^n(\Omega)$, we say that a function~$u \in W^{1, p}_g(\Omega) \cap \W_g^{s, q}(\Omega)$ is a weak solution of the Dirichlet problem~\eqref{dirprobforu} if
\begin{equation} \label{weakDirprob}
\begin{aligned}
& \int_\Omega A(x, Du(x)) \cdot D\varphi(x) \, dx \\
& \hspace{30pt} + \iint_{\mathscr{C}_\Omega} \phi \big( {u(x) - u(y)} \big) \big( {\varphi(x) - \varphi(y)} \big) \frac{B(x, y)}{|x - y|^{n + s q}} \, dx dy = \int_\Omega f \varphi \, dx,
\end{aligned}
\end{equation}
for every~$\varphi \in W^{1, p}_0(\Omega) \cap \W^{s, q}_0(\Omega)$. Moreover, given two functions~$u, v \in W^{1, p}(\Omega) \cap \W^{s, q}(\Omega)$, we say that~$Q u \le Q v$ in~$\Omega$ in the weak sense if
\begin{equation} \label{QuleQvweak}
\begin{aligned}
& \int_{\Omega} \Big( {A(x, Du(x)) - A(x, Dv(x))} \Big) \cdot D\varphi(x) \, dx \\
& \hspace{10pt} + \iint_{\mathscr{C}_\Omega} \Big( {\phi \big( {u(x) - u(y)} \big) - \phi \big( {v(x) - v(y)} \big)} \Big) \left( \varphi(x) - \varphi(y) \right) \frac{B(x, y)}{|x - y|^{n + s q}} \, dx dy \le 0,
\end{aligned}
\end{equation}
for every non-negative function~$\varphi \in W^{1, p}_0(\Omega) \cap \W^{s, q}_0(\Omega)$. By taking respectively~$v \equiv 0$ or~$u \equiv 0$ in the above formulation, we obtain the definition of weak sub- and superharmonic functions for the operator~$Q$ in~$\Omega$, i.e., of weak sub- and supersolutions of~$Q u = 0$ in~$\Omega$. We stress that the left-hand sides of~\eqref{weakDirprob} and~\eqref{QuleQvweak} are well-defined and finite thanks to assumptions~\eqref{ass:A},~\eqref{ass:B},~\eqref{ass:phi} on~$A$,~$B$,~$\phi$, while the finiteness of the right-hand side  of~\eqref{weakDirprob} follows from the embedding of~$W^{1, p}(\Omega)$ into~$L^{\frac{n}{n - 1}}(\Omega)$.

\item Throughout the paper (and mostly in Section~\ref{hopfsec}), we use the prime notation to indicate~$(n - 1)$-dimensional quantities. For instance:~$x'$ denotes the first~$n - 1$ entries of a point~$x \in \R^n$, so that~$x = (x', x_n)$; we write~$B'_r(x_0')$ to denote the~$(n - 1)$-dimensional open ball of radius~$r > 0$ centered at~$x_0' \in \R^{n - 1}$---when the center is omitted, the ball is assumed to be centered at the origin, i.e.,~$B_r' := B_r'(0')$; the notation~$D'$ stands for the gradient of a function defined in (a subset of)~$\R^{n - 1}$, as well as for the horizontal gradient of a function~$H: \Omega \subset \R^n \to \R$, i.e.,~$DH = (D'H, \partial_{x_n} H)$.

\item In the next sections, we denote by~$C$ a constant greater than~$1$ and possibly changing from line to line. Unless otherwise specified, when it appears inside a proof it is assumed to depend on the quantities listed in the corresponding statement.

%
\end{itemize}

\section{Proof of Theorem~\ref{main:thm}}\label{sec:thm}

\noindent
This section is devoted to the proof of Theorem~\ref{main:thm}. In order to make the exposition clearer, we divide it in a few steps.

\subsection*{Step 1: Reduction to nicer outside data}

In this first preliminary step we show that, without loss of generality, the outside datum can be assumed to be compactly supported, of class~$C^{1, \alpha}$ on~$\partial \Omega$, and globally Lipschitz. In order to do this, we first establish the following global~$L^\infty$ estimate for the solution~$u$. We stress that here assumption~\eqref{psq} is not required to hold.

\begin{lemma} \label{Linftyestlem}
Let~$\Omega \subset \subset \Omega' \subset \R^n$ be bounded open sets with~$\partial \Omega$ Lipschitz. Given~$f \in L^n(\Omega)$ and~$g \in W^{1, p}(\Omega) \cap \W^{s, q}(\Omega) \cap L^\infty(\Omega')$, let~$u \in W^{1, p}_g(\Omega) \cap \W^{s, q}_g(\Omega)$ be a weak solution of problem~\eqref{dirprobforu}. Then,~$u \in L^\infty(\Omega)$ and it holds
$$
\| u \|_{W^{1, p}(\Omega)} + \| u \|_{L^\infty(\Omega)} \le C
$$
for some constant~$C > 0$ depending only on~$n$,~$p$,~$q$,~$s$,~$\Lambda$,~$\Omega$, and~$\Omega'$, as well as on~$\| f \|_{L^n(\Omega)}$, $\| g \|_{W^{1, p}(\Omega)}$, $\| g \|_{\W^{s, q}(\Omega)}$, and~$\| g \|_{L^\infty(\Omega')}$.
\end{lemma}

The proof of this result is somewhat standard---it is similar, for instance, to that of~\cite[Proposition~2.1]{min22}. We thus postpone it to Appendix~\ref{app:Linfty}.

Let~$\Omega'' \subset \subset \Omega'$ be an open set with~$\Omega \subset \subset \Omega''$ and~$\eta \in C^\infty_c(\R^n)$ be a smooth cutoff function satisfying~$0 \le \eta \le 1$ in~$\R^n$,~$\eta = 1$ in~$\Omega''$, and~$\supp(\eta) \subset \subset \Omega'$. Set~$\hat{g} \coloneqq \eta g$ and~$\hat{u} \coloneqq \eta u$. Then,~$\hat{u}$ is a weak solution of
$$
\begin{cases}
Q \hat{u} = \hat{f} & \quad \mbox{in } \Omega, \\
\hat{u} = \hat{g} & \quad \mbox{in } \R^n \setminus \Omega,
\end{cases}
$$
where~$\hat{f} = f + \bar{f}$, with
$$
\bar{f}(x) \coloneqq 2 \int_{\R^n \setminus \Omega''} \Big( {\phi \big( {u(x) - \eta(y) g(y)} \big) - \phi \big( {u(x) - g(y)} \big) } \Big) \frac{B(x, y)}{|x - y|^{n + s q}} \, dy \quad \mbox{for } x \in \Omega.
$$

We have that~$\bar{f}\in L^\infty(\Omega)$. To see it, we first observe that, since the Hausdorff distance $\dist(\Om,\,\R^n\setminus\Om'')$ is strictly positive and~$\Om$ is bounded, there exists a constant~$C \ge 1$ depending only on~$\Om$ and~$\Om''$ such that
\begin{equation}\label{bounds:omega}
    C^{-1}(1+|y|)\leq |x-y|\leq C(1+|y|)\quad\text{for all $x\in \Om$ and $y\in \R^n\setminus\Om''$.}
\end{equation}
Using assumptions~\eqref{ass:B}-\eqref{ass:phi}, \eqref{bounds:omega}, and H\"older's inequality, we easily compute
\begin{align*}
\| \bar{f} \|_{L^\infty(\Omega)} & \le C \,  \sup_{x \in \Omega} \left( \int_{\R^n \setminus \Omega''} \frac{|u(x)|^{q - 1} + |g(y)|^{q - 1}}{|x - y|^{n + s q}} \, dy \right) \\
&\le C \left\{ \| u \|_{L^\infty(\Omega)}^{q-1}+ \dashint_{\Omega} \left( \int_{\R^n\setminus\Omega''}\frac{|g(y)|^{q-1}}{(1+|y|)^{n+sq}} \, dy \right) dz \right\}
\\
& \le C \left\{ \| u \|_{L^\infty(\Omega)}^{q-1}+ \| g \|_{L^q(\Omega)}^{q-1} + \left( \iint_{\mathscr{C}_\Omega} \frac{|g(z) - g(y)|^q}{|z - y|^{n + s q}} \, dz dy \right)^{\! \frac{q-1}{q}} \right\},
\end{align*}
for some~$C \ge 1$ depending only on~$n$,~$q$,~$s$,~$\Lambda$,~$\Omega$, and~$\Omega''$. Thus,~$\tilde{f}$ is bounded in~$\Omega$.

Also notice that~$\hat{g} \in C^{1, \alpha}(\partial \Omega) \cap W^{1,\infty}(\R^n)$ and thus, since~$\supp(\hat{g}) \subset \Omega'$, that~$\hat{g} \in W^{a,\chi}(\R^n)$ for all~$a \in (0, 1)$ and~$\chi \ge 1$, with corresponding norms in these spaces bounded only in terms of~$\|g\|_{W^{1,\infty}(\Omega')}$ and~$\| g \|_{C^{1, \alpha}(\partial\Omega)}$.

\subsection*{Step 2: Straightening of the boundary}

We now proceed with the actual proof of Theorem~\ref{main:thm}. To do this, it is convenient to locally straighten the boundary around any given point~$x_0 \in \partial \Omega$. Following the argument of~\cite[Section~5]{min22} (and mostly adopting its notation), we see that there exists a global~$C^{1,\alpha}$-diffeomorphism~$\mathcal{T}$ of~$\rn$ such that
\begin{equation*}
\begin{split}
    \mathcal{T}(x_0)=x_0, & \qquad B_{r_0}^+(x_0)\subset\mathcal{T} \big( {\Om_{3r_0}(x_0)} \big) \subset B_{4r_0}^+(x_0),
    \\
     \G_{r_0}(x_0)&\subset \mathcal{T} \big( {\partial \Om\cap B_{3r_0}(x_0)} \big)\subset\G_{4r_0}(x_0), 
\end{split}
\end{equation*}
for some small radius~$r_0 \in (0, 1]$. Here~$\Omega_r(x_0)=\Omega\cap B_r(x_0)$ and~$\Gamma_r(x_0)=B_r(x_0)\cap \{x_n=0\}$. Write~$\S \coloneqq \T^{-1}$ and~$\mathfrak{c} \coloneqq \big|{\mathcal{J}_{\S}}\big|$, with~$\mathcal{J}_{\S}$ denoting the Jacobian determinant of the inverse~$\S$. Let~$\widetilde{\Om} \coloneqq \mathcal{T}(\Om)$,~$\tg \coloneqq g\circ \S$,~$\tf \coloneqq \mathfrak{c} (f\circ\S)$, and
\begin{equation} \label{utildedef}
\tu \coloneqq u \circ \S.
\end{equation}
It is easy to see that~$\tf \in L^d(\widetilde{\Om})$,~$\tg \in W^{1,\infty}(\R^n)\cap C^{1,\a}(\Gamma_{r_0}(x_0))$, and~$\tu \in W^{1, p}_{\tg}(\widetilde{\Om}) \cap \W_{\tg}^{s, q}(\widetilde{\Om})$. Moreover,~$\tu$ is a weak solution of
\begin{equation} \label{Dirprobforutilde}
\begin{cases}- \dive \widetilde{A}(\cdot, D\tu) + \widetilde{Q}_N \, \tu = \tf & \quad \mbox{in } \widetilde{\Om}, \\
\tu = \tg & \quad \mbox{in } \R^n \setminus \widetilde{\Om},
\end{cases}
\end{equation}
where
$$
\widetilde{A}(x, z) \coloneqq \mathfrak{c}(x)\, A \big( {\S(x), z (D\T \circ \S)(x)} \big)\,\big(D\T \circ \S\big)(x)^T
$$
and
$$
\widetilde{Q}_N \, u(x) \coloneqq 2 \, \PV \int_{\R^n} \phi \left( u(x) - u(y) \right) \widetilde{K}(x, y) \, dy,
$$
with
$$
\widetilde{K}(x, y) \coloneqq \mathfrak{c}(x) \mathfrak{c}(y) \, \frac{B \big( {\S(x), \S(y)} \big)}{|\S(x) - \S(y)|^{n + s q}}.
$$
From assumptions~\eqref{ass:A}-\eqref{ass:B} and the regularity of~$\T$, we infer that
$$
\begin{dcases}
	0<\widetilde{\L}^{-1}\leq \mathfrak{c}(x)\leq \widetilde{\L} & \quad \mbox{for all } x\in\rn,
	\\
	|\mathfrak{c}(x)-\mathfrak{c}(y)|\leq \widetilde{\L} |x-y|^{\alpha} & \quad \mbox{for all } x, y\in B_{r_0}(x_0),
	\\
	\widetilde{K}(x,y) = \widetilde{K}(y, x) & \quad \mbox{for a.a.~} x, y\in\rn,
	\\
	\frac{\widetilde{\L}^{-1}}{|x-y|^{n+s\g}}\leq \widetilde{K}(x,y)\leq \frac{\widetilde{\L}}{|x-y|^{n+s\g}}& \quad \mbox{for a.a.~} x, y\in\rn.
\end{dcases}
$$
and that the~$p$-growth and coercivity conditions are preserved, namely
\begin{equation}\label{ass:fbd}
\begin{dcases}
|\widetilde{A}(x, z)| + |z| |\partial_z \widetilde{A}(x, z)| \le \widetilde{\L} \left( |z|^2 + \mu^2 \right)^{\frac{p - 2}{2}} |z| & \mbox{for all } x \in B_{r_0}^+(x_0),
\\
\big|\widetilde{A}(x,z) - \widetilde{A}(y,z) \big|\leq \widetilde{\L} (|z|^2+\mu^2)^{\frac{p - 1}{2}} |x-y|^{\alpha} & \mbox{for all } x, y \in B_{r_0}^+(x_0), \\
\langle \partial_z \widetilde{A}(x, z) \xi, \xi \rangle \ge \widetilde{\L}^{-1} \left( |z|^2 + \mu^2 \right)^{\frac{p - 2}{2}} |\xi|^2 & \mbox{for all } x \in B_{r_0}^+(x_0), \,  \xi\in\rn,
\end{dcases}
\end{equation}
for every~$z \in \R^n \setminus \{ 0 \}$ and for some constant~$\widetilde{\L} \ge 1$ depending only on~$n$,~$p$,~$\alpha$,~$\Lambda$, and~$\Omega$.

\subsection*{Step 3: Preliminary estimates on~\texorpdfstring{$\tu$}{}}

To prove Theorem~\ref{main:thm}, we need a few lower order estimates on~$\tu$, which mostly follow from the results of~\cite{min22}. In order to obtain them, we first need to introduce the following~``Caccioppoli'' control quantity.

Given any point~$\tx_0\in \Gamma_{r_0/2}(x_0)$, radius~$\vrho \in \left( 0, \frac{r_0}{4} \right]$, and constants~$a,\chi,\gamma$ satisfying
\begin{equation}\label{pr:costs}
	a\in (0,1), \quad \gamma >\max\{p,n\}, \quad \chi>\g, \quad a \chi > n,
\end{equation}
we define
\begin{equation}\label{caccio}
\begin{aligned}
\ccp^+_{\gamma,a,\chi}(\vrho)
& \coloneqq \vrho^{-p}\dashint_{B^+_\vrho(\tx_0)}|\tu-\tg|^p\,dx+\int_{\rn\sm B_\vrho(\tx_0)}\frac{|\tu(y)-(\tu)_{B_\vrho(\tx_0)}|^\g}{|y-\tx_0|^{n+ s \g}}\,dy
\\
& \quad\, + \left( \|\tf\|_{L^n(B^+_\vrho(\tx_0))}^{\frac{p}{p - 1}} + 1 \right) + \bigg( {\dashint_{B^+_\vrho(\tx_0)}|D\tg|^\gamma \,dx} \bigg)^{p/\gamma}
\\
& \quad\, + \bigg( {\vrho^{\chi(a-s)}\int_{B_\vrho(\tx_0)}\dashint_{B_\vrho(\tx_0)}\frac{|\tg(x)-\tg(y)|^\chi}{|x-y|^{n+a \chi}}\,dxdy} \bigg)^{\g/\chi}.
\end{aligned}
\end{equation}

We then have the following preliminary estimates. From now on, we assume the validity of condition~\eqref{psq} and all constants to depend on the quantities declared in the statement of Theorem~\ref{main:thm}.
%
%
%
%

\begin{lemma}
The function~$\tu$ defined by~\eqref{utildedef} belongs to~$C^\beta(\R^n)$ for every~$\beta \in (0, 1)$ and it holds
\begin{equation}\label{holder}
	\| \tu \|_{C^{\b}(\R^n)}\leq C_\beta.
\end{equation}
Moreover, it satisfies
\begin{align}
	\label{ee1}
	\int_{B_t(\tx_0)}\dashint_{B_t(\tx_0)}\frac{|\tu(x)-\tu(y)|^\g}{|x-y|^{n+s \g}}\,dxdy\leq C_\beta \, t^{(\b-s) \g} & \quad \mbox{for all } \beta \in (s, 1) \mbox{ and } t \in \left( 0, \frac{r_0}{4} \right), \\
	\label{ee2}
	\int_{\rn\sm B_t(\tx_0)}\frac{|\tu(y)-(\tu)_{B_t(\tx_0)}|^\g}{|y-\tx_0|^{n+s \g}}\,dy\leq C  & \quad \mbox{for all } t \in \left( 0, \frac{r_0}{4} \right), \\
	\label{ee3}
	\dashint_{B^+_{\vrho/2(\tx_0)}}\big( |D\tu|^2+\mu^2\big)^{p/2}\,dx\leq C_\lambda \, \vrho^{-\l\,p} & \quad \mbox{for all } \lambda > 0 \mbox{ and } \vrho \in \left(0, \frac{r_0}{4}\right).
\end{align}
The constant~$C_\beta$ may also depend on~$\beta$, while~$C_\lambda$ also on~$\lambda$.
\end{lemma}

\begin{proof}
The statement concerning the H\"older regularity of~$\tu$ is the content
of~\cite[Theorem~4 and Proposition~5.1]{min22}---see also~Theorem~6 there and the discussion preceding its statement.
To establish~\eqref{ee1}, it suffices to apply \eqref{holder}. Indeed,
\begin{equation*}
    \begin{split}\int_{B_t(\tx_0)}\dashint_{B_t(\tx_0)}\frac{|\tu(x)-\tu(y)|^\g}{|x-y|^{n+s \g}}\,dxdy & \leq [\tu]_{C^\beta(\R^n)}^q \int_{B_t(\tx_0)}\dashint_{B_t(\tx_0)}\frac{dxdy}{|x-y|^{n+(s-\b) \g}}
    \\
    &\leq C_\beta \int_{B_{2t}(\tx_0)}\frac{dz}{|z|^{n+(s-\b) \g}}\leq C_\beta \,t^{(\b-s) \g},
    \end{split}
\end{equation*}
where we made the change of variables~$z=x-y$ and used the fact~$B_t(\tx_0)-y\subset B_{2t}(\tx_0)$ for every~$y\in B_t(\tx_0)$.

Regarding~\eqref{ee2}, we also use~\eqref{holder} and estimate
\begin{align*}
& \int_{\rn\sm B_t(\tx_0)}\frac{|\tu(y)-(\tu)_{B_t(\tx_0)}|^\g}{|y-\tx_0|^{n+s \g}}\,dxdy \\
& \hspace{30pt} \leq \int_{\rn\sm B_{r_0}(\tx_0)}\frac{|\tu(y)-(\tu)_{B_t(\tx_0)}|^\g}{|y-\tx_0|^{n+s \g}}\,dxdy + 2^{q -1} \int_{B_{r_0}(\tx_0)\sm B_t(\tx_0)}\frac{|\tu(y)-\tu(\tx_0)|^\g}{|y-\tx_0|^{n+s \g}}\,dy \\
& \hspace{30pt} \quad\, + 2^{q - 1} |\tu(\tx_0)-(\tu)_{B_t(\tx_0)}|^\g \int_{B_{r_0}(\tx_0)\sm B_t(\tx_0)}\frac{dy}{|y-\tx_0|^{n+s \g}} \\
& \hspace{30pt} \leq C \left\{ r_0^{- s q} \| \tu \|_{L^\infty(\R^n)}^q + [\tu]_{C^\beta(B_{r_0}(\tx_0))}^q \left( \int_{B_{r_0}}\frac{dz}{|z|^{n+(s-\b) \g}} + t^{\beta q} \int_{\rn\sm B_t}\frac{dz}{|z|^{n+s \g}} \right) \right\} \\
& \hspace{30pt} \leq C_\beta \left( r_0^{ - s \g} + r_0^{(\b-s) \g} + t^{(\b-s) \g} \right),
\end{align*}
for every~$\b\in (s,1)$. By choosing, e.g.,~$\b=(1+s)/2$, we find the desired inequality \eqref{ee2}.

Finally, to prove~\eqref{ee3} we recall the boundary Caccioppoli inequality of~\cite[Lemma 5.1]{min22}: for every~$\gamma,a,\chi$ satisfying~\eqref{pr:costs}, we have
\begin{equation}\label{cac:0}
    \begin{split}
        \dashint_{B^+_{\vrho/2}(\tx_0)}\big(|D\tu|^2+\mu^2  \big)^{p/2}\,dx + \int_{B_{\vrho/2}(\tx_0)}\dashint_{B_{\vrho/2}(\tx_0)}\frac{|\tu(x)-\tu(y)|^\g}{|x-y|^{n+s \g}}\,dxdy\leq C \, \ccp^+_{\gamma,a,\chi}(\vrho).
    \end{split}
\end{equation}
Therefore, to obtain~\eqref{ee3} we only need to estimate each term of~\eqref{caccio}. To this end, by using~\eqref{holder}, the regularity of~$\tg$, and the fact that~$\tu(\tx_0) = \tg(\tx_0)$, we compute
\begin{equation}\label{cac:1}
\begin{split}
    \vrho^{-p}\dashint_{B_{\vrho}^+(\tx_0)}|\tu-\tg|^p\,dx & \leq 2^{p - 1} \vrho^{-p}\dashint_{B_{\vrho}^+(\tx_0)} \Big( {|\tu - \tu(\tx_0)|^p + |\tg - \tg(\tx_0)|^p} \Big) \, dx
    \\
    &\leq 2^{p - 1} \left( [\tu]^p_{C^\beta(B^+_\vrho(\tx_0))}+[\tg]^p_{C^\beta(B^+_\vrho(\tx_0))} \right) \vrho^{(\b-1)\,p} \le C_\b \vrho^{(\b-1)\,p}.
\end{split}
\end{equation}
Clearly, for any fixed constants~$\gamma,a,\chi$ satisfying~\eqref{pr:costs}, we have
\begin{equation}\label{cac:2}
    \left(\dashint_{B^+_\vrho(\tx_0)}|D\tg|^\gamma \,dx \right)^{p/\gamma}\leq C\,[\tg]_{W^{1,\infty}(B^+_\vrho(\tx_0))}^{p}
\end{equation}
and
\begin{equation}\label{cac:3}
\begin{split}
& \left(\vrho^{\chi(a-s)} \int_{B_\vrho(\tx_0)}\dashint_{B_\vrho(\tx_0)}\frac{|\tg(x)- \tg(y)|^\chi}{|x-y|^{n+a \chi}}\,dxdy \right)^{\g/\chi}
    \\
    & \hspace{30pt} \leq \left( \vrho^{\chi(a-s)} [\tg]_{W^{1,\infty}(B_\vrho(\tx_0))}^\chi \int_{B_\vrho(\tx_0)}\dashint_{B_\vrho(\tx_0)}\frac{dxdy}{|x-y|^{n+(a-1)\chi}} \right)^{\g/\chi} \leq C \vrho^{(1-s) \g}.
\end{split}
\end{equation}
%
Therefore, by recalling that~$\vrho \in (0, 1]$ and~$\tf\in L^n(\widetilde{\Om})$, plugging~\eqref{ee2},~\eqref{cac:1},~\eqref{cac:2}, and~\eqref{cac:3} into~\eqref{cac:0}, and choosing~$\beta \ge 1 - \lambda$, we are led to~\eqref{ee3}.
\end{proof}

\subsection*{Step 4: Boundary~\texorpdfstring{$p$}{}-harmonic functions}\label{subsec:pharm}
We will obtain the H\"older continuity of the gradient of~$\tu$ by comparing it to the solution~$\th \in W^{1,p}_{\tu}(B_{\vrho/4}^+(\tx_0))$ of the homogeneous Dirichlet problem
\begin{equation} \label{hom:eq}
\begin{cases}
\dive \widetilde{A}(\tx_0, D\th) = 0 & \quad \mbox{in } B_{\vrho/4}^+(\tx_0), \\
\th = \tu & \quad \mbox{on } \partial B_{\vrho/4}^+(\tx_0).
\end{cases}
\end{equation}
Within this step,~$\varrho$ is a fixed radius in~$\left( 0, r_0 / 4 \right)$.

In the next lemma we collect some useful properties of~$\th$, which are essentially all contained in~\cite[Lemma~5]{lieb88}. We also point out that the existence and uniqueness of~$\th$ is classical---it is mentioned for instance in~\cite{lieb88} and it can be established via the theory of monotone operators (see, e.g.,~\cite[Theorem~26.A]{Z90}).

\begin{lemma} \label{lieblem}
    Let~$\th$ be the solution of problem~\eqref{hom:eq}. Then, there exist constants~$\sigma \in (0,1)$ and~$C > 0$  such that,
    \begin{align}
    \label{bd:minimal1}
    &\dashint_{B_{\vrho/4}^+(\tx_0)}\big( {|D\th|^2+\mu^2} \big)^{p/2}\,dx\leq C \, \dashint_{B_{\vrho/4}^+(\tx_0)}\big( {|D\tu|^2+\mu^2} \big)^{p/2}\,dx,
    \\
    \label{bd:minimal2}
    &\|\th\|_{L^\infty(B_{\vrho/4}^+(\tx_0))}\leq \| \tu \|_{L^\infty(B_{\vrho/4}^+(\tx_0))}, \qquad \osc_{B_{\vrho/4}^+(\tx_0)} \th\leq \osc_{B_{\vrho/4}^+(\tx_0)} \tu,
	\end{align}
	and
	\begin{equation} \label{bdry:holder}
    \osc_{B_{t}^+(\tx_0)} D\th \leq C \left( \frac{t}{\vrho} \right)^{\! \sigma} \left\{ \dashint_{B_{\vrho/4}^+(\tx_0)}\big(|D\th|^2+\mu^2  \big)^{p/2}\,dx +\|\tg\|^p_{C^{1,\alpha}(\Gamma_{r_0}(x_0))} \right\}^{\frac{1}{p}},
    \end{equation}
	for all~$t\in \left(0, \frac{\vrho}{8} \right]$.
\end{lemma}

\begin{proof}
Estimate~\eqref{bdry:holder} is established in~\cite[Lemma~5]{lieb88}, while inequalities~\eqref{bd:minimal2} are an immediate consequence of the weak maximum principle for the elliptic operator~$h \mapsto \dive \widetilde{A}(\tilde{x}_0, Dh)$. Estimate~\eqref{bd:minimal1} can also be obtained by arguing as in the proof of~\cite[Lemma~5]{lieb88}. We provide here a complete proof for the reader's convenience.

By testing the weak formulation of~\eqref{hom:eq} with~$\th - \tu \in W_0^{1, p}(B_{t}^+(\tx_0))$ and taking advantage of estimates~\eqref{ass:fbd}, we find that
\begin{align*}
\int_{B_{t}^+(\tx_0)} \widetilde{A}(\tx_0, D\th) \cdot D\th \, dx & = \int_{B_{t}^+(\tx_0)} \widetilde{A}(\tx_0, D\th) \cdot D\tu \, dx \\
& \le \widetilde{\L} \int_{B_{t}^+(\tx_0)} \big( {|D\th|^2 + \mu^2} \big)^{\! \frac{p - 2}{2}} |D\th| |D\tu| \, dx.
\end{align*}
Using again hypothesis~\eqref{ass:fbd}, we see that~$\widetilde{A}(\tx_0, z) \cdot z \ge \min \left\{ 1, \frac{1}{p - 1} \right\} \widetilde{\L}^{-1} \big( {|z|^2 + \mu^2} \big)^{\! \frac{p - 2}{2}} |z|^2$ for every~$z \in \R^n$, so that
$$
\int_{B_{t}^+(\tx_0)} \widetilde{A}(\tx_0, D\th) \cdot D\th \, dx \ge \frac{\widetilde{\L}^{-1}}{p} \int_{B_{t}^+(\tx_0)} \big( {|D\th|^2 + \mu^2} \big)^{\! \frac{p - 2}{2}} |D\th|^2 \, dx.
$$
Thus,
\begin{equation} \label{bd:minimaltech}
\int_{B_{t}^+(\tx_0)} \big( {|D\th|^2 + \mu^2} \big)^{\! \frac{p - 2}{2}} |D\th|^2 \, dx \le p \, \widetilde{\L}^2 \int_{B_{t}^+(\tx_0)} \big( {|D\th|^2 + \mu^2} \big)^{\! \frac{p - 2}{2}} |D\th| |D\tu| \, dx.
\end{equation}
Now, if~$p \ge 2$ this yields
$$
\int_{B_{t}^+(\tx_0)} |D\th|^p \, dx \le p \, \widetilde{\L}^2 \int_{B_{t}^+(\tx_0)} \big( {|D\th|^2 + \mu^2} \big)^{\! \frac{p - 1}{2}} \big( {|D\tu|^2 + \mu^2} \big)^{\! \frac{1}{2}} \, dx,
$$
which immediately leads to~\eqref{bd:minimal1} after an application of H\"older's inequality. If~$p \in (1, 2)$, we also exploit H\"older's inequality along with the fact that
$$
\frac{ t^{\frac{p}{p - 1}}}{\big( {t^2 + \mu^2} \big)^{\! \frac{(2 - p) p}{2 (p - 1)}}} \le \big( {t^2 + \mu^2} \big)^{\! \frac{p - 2}{2}} t^2, \quad \mbox{for all } t \ge 0,
$$
to deduce from~\eqref{bd:minimaltech} that
\begin{align*}
& \int_{B_{t}^+(\tx_0)} \big( {|D\th|^2 + \mu^2} \big)^{\! \frac{p - 2}{2}} |D\th|^2 \, dx \\
& \hspace{60pt} \le p \, \widetilde{\L}^2 \Bigg( {\int_{B_{t}^+(\tx_0)} \frac{ |D\th|^{\frac{p}{p - 1}}}{\big( {|D\th|^2 + \mu^2} \big)^{\! \frac{(2 - p) p}{2 (p - 1)}}} \, dx} \Bigg)^{\! \frac{p - 1}{p}} \bigg( {\int_{B_{t}^+(\tx_0)} |D\tu|^p \, dx} \bigg)^{\! \frac{1}{p}} \\
& \hspace{60pt} \le 2 \widetilde{\L}^2 \Bigg( {\int_{B_{t}^+(\tx_0)} \big( {|D\th|^2 + \mu^2} \big)^{\! \frac{p - 2}{2}} |D\th|^2 \, dx} \Bigg)^{\! \frac{p - 1}{p}} \bigg( {\int_{B_{t}^+(\tx_0)} |D\tu|^p \, dx} \bigg)^{\! \frac{1}{p}}.
\end{align*}
This gives
$$
\int_{B_{t}^+(\tx_0)} \big( {|D\th|^2 + \mu^2} \big)^{\! \frac{p - 2}{2}} |D\th|^2 \, dx \le 2^p \widetilde{\L}^{2 p} \int_{B_{t}^+(\tx_0)} |D\tu|^p \, dx \le 4 \widetilde{\L}^{2 p} \int_{B_{t}^+(\tx_0)} \big( {|D\tu|^2 + \mu^2} \big)^{p/2} \, dx,
$$
which, together with the trivial estimate
$$
\int_{B_{t}^+(\tx_0)} \big( {|D\th|^2 + \mu^2} \big)^{\! \frac{p - 2}{2}} \mu^2 \, dx \le \int_{B_{t}^+(\tx_0)} \mu^p \le \int_{B_{t}^+(\tx_0)} \big( {|D\tu|^2 + \mu^2} \big)^{p/2} \, dx,
$$
readily yields~\eqref{bd:minimal1}. The proof of~\eqref{bd:minimal1} is thus complete.
\end{proof}

Next, we consider the function~$\tw \coloneqq \tu-\th\in W^{1,p}_0(B_{\vrho/4}^+(\tx_0))$ and extend it to~$\rn$ by setting~$\tw\equiv 0$ in~$\rn\sm B_{\vrho/4}^+(\tx_0)$. Note that this new function~$\tw$ belongs to~$W^{s, q}(\R^n)$ and thus to~$\W^{s,q}_0(B_{\vrho/4}^+(\tx_0))$. This is a consequence of its boundedness and of the fact that~$p \ge s q$---see, e.g.,~\cite[Lemma~2.4]{min22},~\cite[Lemma 5.1]{hitch}, and also the discussion at the beginning of the proof of~\cite[Lemma~5.2]{min22}. Furthermore, by~\eqref{holder} and~\eqref{bd:minimal2}, we infer that
\begin{equation}\label{osc:w}
\|\tw\|_{L^\infty(B_{\vrho/4}^+(\tx_0))}\leq  \osc_{B_{\vrho/4}^+(\tx_0)}\tu + \osc_{B_{\vrho/4}^+(\tx_0)}\th\leq 2 \osc_{B_{\vrho/4}^+(\tx_0)} \tu\leq 2 \,[\tu]_{C^\beta(B_{\vrho/4}^+(\tx_0))}\,\vrho^\b\leq C_\b \, \vrho^\b,
\end{equation}
for every~$\b\in (0,1)$ and for some constant~$C_\b>0$ depending also on~$\b$.

In order to continue with the proof of Theorem~\ref{main:thm}, we need to introduce a few more important quantities and recall a couple of useful inequalities. We set
\begin{equation*}
    V_\mu(z)\coloneqq \big( |z|^2+\mu^2\big)^{\frac{p-2}{4}} z \quad\text{for } z\in\rn.
\end{equation*}
It is not hard to see that there exists a constant~$C > 0$, depending only on~$n$,~$p$, and~$\widetilde{\Lambda}$, for which
$$
|V_\mu(z_1) - V_\mu(z_2)|^2 \le C \left( \widetilde{A}(\tx_0, z_1) - \widetilde{A}(\tx_0, z_2) \right) \cdot \left( z_1 - z_2 \right) \quad \mbox{for all } z_1, z_2 \in \R^n.
$$
This is a consequence of the structural hypotheses~\eqref{ass:fbd}---see, e.g.,~\cite[(2.10)]{min22}. As a consequence, defining~$\Nu^2 \coloneqq |V_\mu(D\tu)-V_\mu(D\th)|^2$, we see that
\begin{equation} \label{Vtildeest}
    \Nu^2\leq C \left( \widetilde{A}(\tx_0,D\tu)- \widetilde{A}(\tx_0,D\th)\right) \cdot D\tw \quad \mbox{a.e.~in } \R^n.
\end{equation}
On the other hand, by using~\cite[(2.9)]{min22} and H\"older's inequality it follows that
\begin{equation}\label{stand:nu}
\begin{aligned}
& \frac{1}{C} \, \dashint_{B_{\vrho/4}^+(\tx_0)} |D\tu-D\th|^p\,dx \\
& \hspace{10pt} \le \begin{dcases}\dashint_{B_{\vrho/4}^+(\tx_0)} \Nu^2\,dx & \,\, \text{if } p\geq 2, \\
\left( \dashint_{B_{\vrho/4}^+(\tx_0)} \Nu^2\,dx \right)^{\! \frac{p}{2}} \left(\dashint_{B_{\vrho/4}^+(\tx_0)}  \left( |D\tu|^2+|D\th|^2+\mu^2\right)^{p/2} dx  \right)^{\!\! \frac{2 - p}{2}} & \,\, \mbox{if } p \in (1, 2).
\end{dcases}
\end{aligned}
\end{equation}

Using these inequalities we may quantify the closeness of the gradients of~$\tu$ and~$\th$, as described by the following result.

\begin{lemma} \label{Du-Dhsmall}
    Let~$\tu$ and~$\th$ be the functions defined in~\eqref{utildedef} and~\eqref{hom:eq}, respectively. Then there exist constants~$C>0$ and~$\bar{\sigma}\in (0,1)$ such that,
    \begin{equation}\label{camp1}
        \dashint_{B_{\vrho/4}^+(\tx_0)}|D\tu-D\th|^p\,dx\leq C \vrho^{\bar{\sigma} p}.
    \end{equation}
\end{lemma}

\begin{proof}
First we notice that, by definition of~$\tw$,~\eqref{ee3}, and~\eqref{bd:minimal1}, it holds
\begin{equation}\label{start:Dw}
    \dashint_{B_{\vrho/4}^+(\tx_0)}|D\tw|^p\,dx\leq C\, \dashint_{B_{\vrho/4}^+(\tx_0)}\big(|D\tu|^2+\mu^2\big)^{p/2}\,dx\leq C_\lambda \, \vrho^{-\l\,p}
\end{equation}
for every~$\l>0$ and for some constant~$C_\lambda > 0$ depending also on~$\lambda$. By plugging~$\tw$ in the weak formulations of both~\eqref{Dirprobforutilde} and~\eqref{hom:eq}, taking advantage of~\eqref{Vtildeest}, and arguing as in the proof of~\cite[Lemma 5.2]{min22}, we estimate
\begin{equation} \label{start:nu01}
\dashint_{B_{\vrho/4}^+(\tx_0)}\Nu^2\,dx\leq C \Big( I_1 + I_2 + I_3 + I_4 \Big),
\end{equation}
where
\begin{equation}\label{start:nu02}
\begin{aligned}
I_1 & \coloneqq \vrho^{\alpha} \dashint_{B_{\vrho/4}^+(\tx_0)}\big( |D\tu|^2+\mu^2\big)^{(p-1)/2}\,|D\tw|\,dx, \\
I_2 & \coloneqq \dashint_{B_{\vrho/4}^+(\tx_0)}|\tf \tw|\,dx, \\
I_3 & \coloneqq \int_{B_{\vrho/2}(\tx_0)} \dashint_{B_{\vrho/2}(\tx_0)}\frac{|\tu(x)-\tu(y)|^{\g-1}|\tw(x)-\tw(y)|}{|x-y|^{n+s \g}}\,dxdy, \\
I_4 & \coloneqq \int_{\rn\sm B_{\vrho/2}(\tx_0)} \left( \dashint_{B_{\vrho/2}(\tx_0)} \frac{|\tu(x)-\tu(y)|^{\g-1}|\tw(x)|}{|x-y|^{n+s \g}}\,dx \right) dy.
\end{aligned}
\end{equation}
    
By H\"older's inequality and~\eqref{start:Dw}, we get
\begin{equation}\label{start:nu1}
I_1 \leq C \vrho^{\alpha} \left( \dashint_{B_{\vrho/4}^+(\tx_0)} \big(|D\tu|^2+\mu^2 \big)^{p/2} \, dx \right)^{\! \frac{p-1}{p}} \left(\dashint_{B_{\vrho/4}^+(\tx_0)}|D\tw|^p\,dx \right)^{\! 1/p} \leq C_\lambda \, \vrho^{\alpha - \l p},
\end{equation}
for every~$\l > 0$. Next, by using H\"older and Sobolev inequalities---recall that~$\tw$ vanishes on the boundary of~$B_{\vrho/4}^+(\tx_0)$---together with~\eqref{start:Dw}, we infer
\begin{equation}\label{start:nu2}
    \begin{aligned}
        I_2 & \leq \frac{C}{\vrho^n} \, \|\tf\|_{L^n(B_{\vrho/4}^+(\tx_0))}\|\tw\|_{L^{\frac{n}{n-1}}(B_{\vrho/4}^+(\tx_0))} \leq \frac{C}{\vrho^n} \, \|\tf\|_{L^n(B_{\vrho/4}^+(\tx_0))}\|D\tw\|_{L^1(B_{\vrho/4}^+(\tx_0))}
        \\
        &\leq C \vrho^{1-\frac{n}{d}}\|\tf\|_{L^d(B_{\vrho/4}^+(\tx_0))} \left(\dashint_{B^+_{\vrho/4}(\tx_0)}|D\tw|^p\,dx  \right)^{\! 1/p} \le C_\lambda \, \vrho^{1-\frac{n}{d} - \l},
    \end{aligned}
\end{equation}
for every~$\l > 0$. We now take advantage of H\"older's inequality once again, estimate~\eqref{ee1}, and the interpolation inequality of~\cite[Lemma 2.4]{min22} in the ball~$B_{\varrho/2}(\tx_0)$ to find
\begin{align*}
I_3 & \leq  C \left( \int_{B_{\vrho/2}(\tx_0)} \dashint_{B_{\vrho/2}(\tx_0)} \! \! \frac{|\tu(x)-\tu(y)|^{\g}}{|x-y|^{n+s \g}}\,dxdy \right)^{\!\! \frac{q - 1}{q}} \! \left( \int_{B_{\vrho/2}(\tx_0)} \dashint_{B_{\vrho/2}(\tx_0)} \! \! \frac{|\tw(x)-\tw(y)|^{\g}}{|x-y|^{n+s \g}}\,dxdy \right)^{\! \frac{1}{\g}}
\\
&\leq C_\b \, \vrho^{(\b-s)(\g-1)+\vartheta-s} \, \|\tw\|_{L^{\infty}(B_{\vrho/4}^+(\tx_0))}^{1-\vartheta} \left( \dashint_{B_{\vrho/4}^+(\tx_0)}|D\tw|^p \, dx \right)^{\! \vartheta/p}
\\
&\leq C_{\b, \l} \, \vrho^{(\b-s)(\g-1)+\vartheta-s+\b(1-\vartheta)-\vartheta\,\l},
\end{align*}
for every~$\b \in (s, 1)$ and~$\l>0$, where~$C_{\b,\l}$ is a constant possibly depending on~$\l$ and~$\beta$, and~$\vartheta$ is defined by
\begin{equation*}
	\vartheta \coloneqq \begin{cases}
		s\quad & \text{if } \g>p,
		\\
		1\quad & \text{if } \g\leq p.
	\end{cases}
\end{equation*}
Note that in the last inequality we applied~\eqref{osc:w} and~\eqref{start:Dw}. From this estimate and the definition of~$\vartheta$, we deduce in particular that
\begin{equation}\label{start:nu4}
I_3 \leq C_{\b, \l} \, \vrho^{(\b-s)(\g-1)-\l},
\end{equation}
for every~$\b \in (s, 1)$ and~$\l>0$.

Finally, we estimate~$I_4$. Since~$\tw$ is supported in~$B_{\vrho/4}^+(\tx_0)$ and it holds
\begin{equation*}
    \frac{|y-\tx_0|}{|x-y|}\leq 2\quad\text{for every } x\in B_{\vrho/4}^+(\tx_0) \mbox{ and } y\in \rn\sm B_{\vrho/2}(\tx_0),
\end{equation*}
we have
    \begin{align}
\nonumber
        I_4 & \leq C \int_{\rn\sm B_{\vrho/2}(\tx_0)} \Bigg( \dashint_{B_{\vrho/4}^+(\tx_0)} \! \! \frac{\Big(|\tu(x)-(\tu)_{B_{\vrho/2}(\tx_0)}|^{\g-1}+|\tu(y)-(\tu)_{B_{\vrho/2}(\tx_0)}|^{\g-1}\Big)\,|\tw(x)|}{|y-\tx_0|^{n+s \g}}\,dx \Bigg) dy
        \\
\label{start:nu5}
        &\leq C \vrho^{-s \g} \dashint_{B_{\vrho/4}^+(\tx_0)}|\tu(x)-(\tu)_{B_{\vrho/2}(\tx_0)}|^{\g-1}\,|\tw(x)|\,dx
        \\
\nonumber
        &\quad + C \left( \int_{\rn\sm B_{\vrho/2}(\tx_0)}\frac{|\tu(y)-(\tu)_{B_{\vrho/2}(\tx_0)}|^{\g-1}}{|y-\tx_0|^{n+s \g}} \, dy \right) \left( \dashint_{B_{\vrho/4}^+(\tx_0)}|\tw(x)|\,dx \right) \! ,
    \end{align}
where in the second inequality we used that
\begin{equation}\label{int:tail}
    \int_{\rn\sm B_{\vrho/2}(\tx_0)}\frac{dy}{|y-\tx_0|^{n+s \g}}\leq C \vrho^{-s \g}.
\end{equation}
By H\"older's inequality,~\eqref{holder}, and~\eqref{osc:w}, we obtain
\begin{align*}
        & \vrho^{-s \g} \dashint_{B_{\vrho/4}^+(\tx_0)}|\tu(x)-(\tu)_{B_{\vrho/2}(\tx_0)}|^{\g-1}\,|\tw(x)|\,dx
        \\
        & \hspace{40pt} \leq C \vrho^{-s \g} \left( \dashint_{B_{\vrho/4}^+(\tx_0)}|\tu-(\tu)_{B_{\vrho/2}(\tx_0)}|^{\g} \, dx \right)^{\! \! \frac{q - 1}{q}} \left( \dashint_{B_{\vrho/4}^+(\tx_0)}|\tw|^\g \, dx \right)^{\! \frac{1}{\g}}
        \\
        & \hspace{40pt} \leq C \vrho^{-s\g+\b(\g-1)}[\tu]^{\g-1}_{C^{\b} (B_{\vrho/2}(\tx_0))} \,\|\tw\|_{L^\infty(B_{\vrho/4}^+(\tx_0))} \leq C_\b \, \vrho^{(\b-s)\g},
\end{align*}
whereas, by H\"older's inequality,~\eqref{int:tail},~\eqref{ee2}, and~\eqref{osc:w}, we get
\begin{align*}
        & \left( \int_{\rn\sm B_{\vrho/2}(\tx_0)} \frac{|\tu(y)-(\tu)_{B_{\vrho/2}(\tx_0)}|^{\g-1}}{|y-\tx_0|^{n+s \g}} \, dy \right) \left( \dashint_{B_{\vrho/4}^+(\tx_0)}|\tw(x)|\,dx \right)
        \\
        & \hspace{50pt} \leq C \vrho^{-s} \left(\int_{\rn\sm B_{\vrho/2}(\tx_0)}\frac{|\tu(y)-(\tu)_{B_{\vrho/2}(\tx_0)}|^{\g}}{|y-\tx_0|^{n+s\g}} \, dy \right)^{\! 1-\frac{1}{\g}}\|\tw\|_{L^\infty(B_{\vrho/4}^+(\tx_0))}
        \\
        & \hspace{50pt} \leq C_\b \,\vrho^{\b-s}.
\end{align*}
By inserting these two inequalities into~\eqref{start:nu5} and recalling that~$\vrho \in (0, 1]$ and~$\g>1$, we find that
\begin{equation}\label{start:nu8}
I_4 \leq C_\b \,\vrho^{\b-s},
\end{equation}
for every~$\b \in (s, 1)$.

All in all, by plugging~\eqref{start:nu1},~\eqref{start:nu2},~\eqref{start:nu4}, and~\eqref{start:nu8} into~\eqref{start:nu01}-\eqref{start:nu02}, we obtain the integral inequality
\begin{equation}\label{go}
    \dashint_{B_{\vrho/4}^+(\tx_0)}\Nu^2\,dx\leq C_{\b, \l} \Big( {\vrho^{\alpha-\l p}+\vrho^{1-\frac{n}{d}-\l}+\vrho^{(\b-s)(\g-1)-\l}+\vrho^{\b-s}} \Big),
\end{equation}
for every~$\b \in (s, 1)$ and~$\l > 0$. We now choose the constants~$\l$ and~$\b$ as follows:
\begin{equation*}
    \b \coloneqq\frac{1+s}{2} \quad \text{and} \quad \l \coloneqq\min \left\{ \frac{\alpha}{2p}, \frac{1}{2}\Big(1-\frac{n}{d}\Big), \frac{(1-s)(\g-1)}{4} \right\},
\end{equation*}
so that~\eqref{go} becomes just
\begin{equation}\label{go1}
     \dashint_{B_{\vrho/4}^+(\tx_0)}\Nu^2\,dx\leq C\vrho^{\sigma_0 \, p},
\end{equation}
with~$\sigma_0 \coloneqq \frac{1}{p}\min \left\{ \frac{\alpha}{2}, \frac{1}{2}\big(1-\frac{n}{d}\big), \frac{(1-s)(\g-1)}{4}, \frac{1 - s}{2} \right\}$.

We are now in position to conclude, using~\eqref{go1} in combination with~\eqref{stand:nu}. When~$p\geq 2$, estimate~\eqref{camp1} follows immediately with $\bar{\sigma}=\sigma_0$. On the other hand, when~$p \in (1, 2)$ we estimate the second factor in~\eqref{stand:nu} through~\eqref{bd:minimal1} and~\eqref{ee3}, obtaining
$$
\dashint_{B_{\vrho/4}^+(\tx_0)}|D\tu - D\th|^p \, dx\leq C_\l \, \vrho^{\frac{\sigma_0 p - (2 - p)\l}{2} \, p} \quad \mbox{for every } \l > 0.
$$
Therefore, by choosing~$\l \coloneqq \frac{\sigma_0 \, p}{2 (2-p)}$, we obtain the desired estimate~\eqref{camp1} with~$\bar{\sigma}=\frac{\sigma_0\,p}{4}$. The proof is thus complete.
\end{proof}

\subsection*{Step 5: Conclusion}

Having Lemma~\ref{Du-Dhsmall}, we are now ready to prove a Campanato type boundary estimate and thus, with it, Theorem~\ref{main:thm}.

\begin{proposition}
	Let~$\tu$ be the function defined in~\eqref{utildedef}. Then, there exist a radius~$\vrho_0\in (0,1)$ and costants~$C>0$,~$\sigma_1\in (0,1)$  such that,
    \begin{equation}\label{bd:camp}
        \sup_{\tx_0\in \G_{r_0/2}} \, \dashint_{B^+_{\vrho}(\tx_0)}\big| D\tu-(D\tu)_{B^+_{\vrho}(\tx_0)} \big|^p\,dx\leq C \vrho^{\sigma_1\,p}\quad\text{for every } \vrho \in (0, \vrho_0].
    \end{equation}
\end{proposition}

\begin{proof}
Let~$t \in \left( 0, \frac{\vrho}{8} \right]$, with~$\vrho \in \left( 0, \frac{r_0}{4} \right]$. For every~$\tx_0 \in \G_{r_0/2}$, we have
\begin{align*}
& \dashint_{B_t^+(\tx_0)}\big| D\tu - (D\tu)_{B_t^+(\tx_0)} \big|^p \, dx
\\
& \hspace{50pt} \leq 2^{p-1} \dashint_{B_t^+(\tx_0)} \big| D\tu-D\th\big|^p \, dx + 4^{p-1} \dashint_{B_t^+(\tx_0)} \big|D\th-(D\th)_{B_t^+(\tx_0)}  \big|^p \, dx
\\
& \hspace{50pt} \quad +4^{p-1} \big|(D\tu)_{B_t^+(\tx_0)}-(D\th)_{B_t^+(\tx_0)}  \big|^p
\\
& \hspace{50pt} \leq C \left\{ \dashint_{B_t^+(\tx_0)} \big|  D\tu-D\th\big|^p \, dx + \dashint_{B_t^+(\tx_0)} \big|D\th-(D\th)_{B_t^+(\tx_0)}  \big|^p \, dx \right\}
\\
& \hspace{50pt} \leq C \left\{ \left( \frac{\vrho}{t} \right)^n \dashint_{B_{\vrho/4}^
	+(\tx_0)} \big|  D\tu-D\th\big|^p \, dx+ \bigg( {\osc_{B_{t}^+(\tx_0)} D\th} \bigg)^p \right\}.
\end{align*}
Recalling~\eqref{bdry:holder},~\eqref{camp1},~\eqref{bd:minimal1}, and~\eqref{ee3}, this yields
\begin{align*}
& \dashint_{B_t^+(\tx_0)}\big| D\tu - (D\tu)_{B_t^+(\tx_0)} \big|^p \, dx \\
& \hspace{5pt} \le C  \left\{ \bigg( {\frac{\vrho}{t}} \bigg)^n \vrho^{\bar{\sigma} p} + \bigg( {\frac{t}{\vrho}} \bigg)^{\sigma p} \left( \dashint_{B_{\vrho/4}^+(\tx_0)}\big(|D\th|^2+\mu^2  \big)^{p/2}\,dx +\|\tg\|^p_{C^{1,\alpha}(\Gamma_{r_0}(x_0))} \right) \right\} \\
& \hspace{5pt} \le C_\lambda \left\{ \bigg( {\frac{\vrho}{t}} \bigg)^n \vrho^{\bar{\sigma} p} + \bigg( {\frac{t}{\vrho}} \bigg)^{\sigma p} \left( \vrho^{-\l p} + \|\tg\|^p_{C^{1,\a_b}(\Gamma_{r_0}(x_0))} \right) \right\} \le C_\lambda \left\{ \bigg( {\frac{\vrho}{t}} \bigg)^n \vrho^{\bar{\sigma} p} + \bigg( {\frac{t}{\vrho}} \bigg)^{\sigma p} \vrho^{-\l p} \right\},
\end{align*}
for every~$\lambda > 0$. By choosing~$t \coloneqq \frac{\vrho^{1+ \frac{\bar{\sigma} p}{2n}}}{8}$ and~$\l \coloneqq \frac{\sigma \, \bar{\sigma} \, p}{4n}$, we then obtain
$$
    \dashint_{B_t^+(\tx_0)}\big| D\tu- (D\tu)_{B_t^+(\tx_0)} \big|^p \, dx\leq C \, t^{\sigma_1 \, p} \quad \mbox{for every } t \in (0, \vrho_0),
$$
with~$\vrho_0 \coloneqq \frac{1}{8} \! \left( \frac{r_0}{4} \right)^{1+ \frac{\bar{\sigma} p}{2n}}$ and~$\sigma_1 \coloneqq \min \left\{ \frac{n \bar{\sigma}}{2n+\bar{\sigma}p}, \frac{\sigma \, \bar{\sigma} \, p}{2(2n+\bar{\sigma} p)} \right\}$. This concludes the proof of~\eqref{bd:camp}, up to relabeling~$t$ as~$\vrho$.
\end{proof}

\begin{proof}[Proof of Theorem \ref{main:thm}]
By combining the interior Campanato estimate of~\cite[Theorem~5]{min22} and the boundary estimate~\eqref{bd:camp}, the result follows via a standard covering argument and Campanato's characterization of H\"older spaces~\cite{camp,camp1}--see also~\cite[Section 5]{giaq}.
\end{proof}

\section{A weak comparison principle} \label{sec:weak}

\noindent
The aim of this very brief section is to establish a weak comparison principle for the operator~$Q$, which will be used shortly to prove Theorem~\ref{Hopflemma}. The precise statement is as follows.

\begin{proposition} \label{WCPprop}
Let~$\Omega \subset \R^n$ be a bounded open set with Lipschitz boundary. Assume that~$A$,~$B$, and~$\phi$ satisfy hypotheses~\eqref{ass:A},~\eqref{ass:B}, and~\eqref{ass:phi}. Let~$u, v \in W^{1, p}(\Omega) \cap \W^{s, q}(\Omega)$ be satisfying~$Q u \le Q v$ in~$\Omega$ in the weak sense. If~$u \le v$ in~$\R^n \setminus \Omega$, then~$u \le v$ in~$\Omega$ as well.
\end{proposition}
\begin{proof}
By plugging~$\varphi = (u - v)_+$ in the weak formulation~\eqref{QuleQvweak} and observing that, by the monotonicity of~$\phi$,
$$
\Big( {\phi \big( {u(x) - u(y)} \big) - \phi \big( {v(x) - v(y)} \big)} \Big) \Big( {\big( {u(x) - v(x)} \big)_+ - \big( {u(y) - v(y)} \big)_+} \Big) \ge 0,
$$
for a.e.~$x, y \in \R^n$, we obtain that
$$
\int_{\Omega_+} \Big( {A(x, Du(x)) - A(x, Dv(x))} \Big) \cdot \big( {Du(x) - Dv(x)} \big) \, dx \le 0,
$$
where~$\Omega_+ \coloneqq \left\{ x \in \Omega : u(x) > v(x) \right\}$. From the third line of assumption~\eqref{ass:A} on~$A$, it is immediate to deduce that the integrand above is non-negative and vanishes only at those points~$x \in \Omega_+$ where~$Du(x) = Dv(x)$---see, e.g.,~\cite[Lemma~2.1 and Theorem~1.2]{dam}.

Therefore, we conclude that~$D u = D v$~in~$\Omega_+$, and thus that~$u \le v$ in~$\Omega$.
\end{proof}

\section{Proof of Theorem~\ref{Hopflemma}} \label{hopfsec}

\noindent
In this section we establish Theorem~\ref{Hopflemma}, whose proof will be divided into a few steps. Note that, given~$r,\rho>0$, we write~$\mathcal{C}_{r,\,\rho}^+ \coloneqq B'_r\times (0,\rho)$ and~$\mathcal{C}_r^+=\mathcal{C}_{r,r}^+$.

\subsection*{Step 1: Straightening of the boundary}

Differently from Section~\ref{sec:thm}, here we need to consider a more specific diffeomorphism of~$\R^n$ in order to pointwise evaluate the operator~$Q$.

Up to a rigid movement, we may assume that~$x_0 = 0$ and~$\nu(0) = - e_n$. Therefore, since~$\partial \Omega$ is of class~$C^{1, \alpha}$, there exist a radius~$R \in (0, 1)$ and a function~$h \in C^{1, \alpha}(\R^{n - 1})$ vanishing outside of~$B'_{4 R}$, satisfying
\begin{equation}\label{h00}
h(0')=0, \quad D'h(0')=0',
\end{equation}
and such that
\begin{equation} 
\begin{aligned} \label{Omegasupgraph}
\Omega \cap B_{2 R} & = \Big\{ {(x', x_n) \in B_{2 R} : x_n > h(x')} \Big\}, \\
\partial\Omega \cap B_{2 R} & = \Big\{ {(x', x_n) \in B_{2 R} : x_n = h(x')} \Big\}.
\end{aligned}
\end{equation}
Moreover, by suitably modifying~$h$ in~$B'_{4 R} \setminus B'_{3 R}$, we may also assume that
\begin{equation} \label{halflapofhis0}
(-\Delta)^{\frac{1}{2}} h(0) = \pi^{-\frac{n}{2} } \Gamma \left( \frac{n}{2} \right) \PV \int_{\R^{n-1}} \frac{h(0) - h(z')}{|z'|^n} \, dz' = 0.
\end{equation}
Indeed, it suffices to replace~$h$ by the function~$h + \ell \phi$, for an arbitrary~$\phi \in C^{\infty}_c(B'_{4R}\setminus B'_{3R})$ and with~$\ell \coloneqq - \Big( {(-\Delta)^{\frac{1}{2}} \phi(0)} \Big)^{-1} (-\Delta)^{\frac{1}{2}} h(0)$.

We straighten the boundary of~$\Omega$ inside~$B_{2 R}$ via a suitable diffeomorphism~$\T: \R^n \to \R^n$ globally of class~$C^{1, \alpha}$, but actually smooth inside~$\Omega \cap B_{2 R}$. In order to do this, we first consider a \emph{nice} extension of~$h$ to the whole space~$\R^n$.

\begin{lemma}\label{lemma:ext}
    Given~$\alpha \in (0, 1)$, let~$h\in C^{1,\a}(\R^{n-1})$ be a compactly supported function satisfying~\eqref{h00} and~\eqref{halflapofhis0}. Then, there exists a function~$H\in C^{1,\a}(\R^n) \cap C^\infty(\R^n_+)$ such that~$H(x',0)=h(x')$ for all~$x'\in \R^{n-1}$,~$DH(0)=0$,
\begin{equation} \label{ssstimaDH}
\| H \|_{C^{1, \a}(\R^n)} \leq C \, \| h \|_{C^{1,\a}(\R^{n-1})}
\end{equation}
and
    \begin{equation}\label{stime:H1}
        \big| D^2 H(y',y_n) \big| \leq C\,[D'h]_{C^{\a}(\R^{n-1})}\,y_n^{\a-1} \quad \text{for all } (y',y_n)\in \R^n_+,
    \end{equation}
    for some constant~$C>0$ depending only on~$n$ and~$\a$.
\end{lemma}

We take as~$H$ a suitable~$C^{1, \alpha}(\R^n)$-continuation of the harmonic extension of~$h$ to the upper half-space. The proof of Lemma~\ref{lemma:ext} is then rather natural and follows from the Poisson representation for~$H$. For this reason, we postpone it to Appendix~\ref{app:Hextproof} and resume here the proof of Theorem~\ref{Hopflemma}.

Let
\begin{equation}\label{eta}
    \eta \coloneqq \left( 1 + 2\,\| D H \|_{L^\infty(\R^n)} \right)^{-1}\,,
\end{equation}
and define~$\S: \R^n \to \R^n$ by setting
$$
\S(y', y_n) \coloneqq \big( {y', y_n + H(y', \eta\, y_n)} \big) \quad \mbox{for all } (y', y_n) \in \R^n.
$$
Clearly, its Jacobian matrix is given by
\begin{equation} \label{matrS}
D\S(y',y_n)= \left(\begin{array}{c|c}
 \vphantom{\Bigg|} \mbox{Id}_{n - 1}
  & 0' \\
\hline
  \vphantom{\bigg|} D'H(y', \eta \, y_n)^T &
  1 + \eta \, \partial_{y_n} H(y', \eta \, y_n)
\end{array}\right) \! .
\end{equation}
Denoting with~$\mathfrak{c} \coloneqq \mathcal{J}_\S$ its Jacobian determinant, we have
\begin{equation*}
    \mathfrak c(y)=\partial_{y_n}\S^n(y',y_n)=1+\eta\,\partial_{y_n}H(y',\eta\,y_n)\,,
\end{equation*}
 so that~\eqref{ssstimaDH},~\eqref{stime:H1}, and~\eqref{eta} entail
\begin{equation}\label{detundercontrol}
    \mathfrak{c}(y) \in \left[ \frac{1}{2}, \frac{3}{2} \right]  \quad \mbox{for all } y \in \R^n,
\end{equation}
 and
\begin{equation} \label{detundercontrol1}
\begin{dcases}
\|\mathfrak c\|_{C^\a(\R^n)}\leq C  \\
|D\mathfrak{c}(y)| \le C\,y_n^{\a-1}  \quad \mbox{for all } y \in \R^n_+.
\end{dcases}
\end{equation}
In this step,~$C$ indicates a constant depending only on~$n$,~$\alpha$, and~$\| h \|_{C^{1, \alpha}(\R^{n - 1})}$. Therefore, it is immediate to see that~$\S$ is a~$C^{1, \alpha}$-diffeomorphism of~$\R^n$ onto itself, such that
\begin{equation} \label{Psimapprop}
\S(\R^n_+) = \Big\{ {(x', x_n) \in \R^n : x_n > h(x')} \Big\}, \quad \S(\partial \R^n_+) = \Big\{ {(x', x_n) \in \R^n : x_n = h(x')} \Big\}.
\end{equation}
In particular, setting $\T=\S^{-1}$, explicit computations show that 
\begin{equation}\label{matrT}
    (D\T\circ\S)(y)= D\S(y)^{-1} = 
\left(\begin{array}{c|c}
 \vphantom{\Bigg|} \mbox{Id}_{n - 1}
  & 0' \\
\hline
  \vphantom{\Bigg|} \dfrac{-D'H(y',\eta\,y_n)^T}{1 + \eta \, \partial_{y_n}H(y',\eta\,y_n)} &
  \dfrac{1}{1 + \eta \, \partial_{y_n} H(y', \eta \, y_n)}
\end{array}\right) \! .
\end{equation}
Then, from~\eqref{ssstimaDH},~\eqref{eta},~\eqref{matrS}, and~\eqref{matrT}, we infer that
 \begin{equation}\label{bound:DST}
     \| D\S \|_{C^{\a}(\R^n)} + \| D\T\circ\S \|_{C^{\a}(\R^n)} \leq C.
 \end{equation}
In particular, these estimates yield the global Lipschitz bounds
\begin{equation} \label{SLip}
C^{-1} |y-z| \leq |\S(y)-\S(z)|\leq C\,|y-z | \quad \mbox{for all } y, z \in \R^n.
\end{equation}
Since~$DH(0)=0$, we have that~$D\S(0)=(D\T\circ\S)(0) = \mbox{Id}_n$ and thus, by~\eqref{bound:DST},
$$
\big| {D\S(y)-\mathrm{Id}_n} \big| + \big| {(D\T\circ\S)(y)-\mathrm{Id}_n} \big| \leq C\,|y|^\a \quad \mbox{for all }y\in \R^n.
$$
This also implies that
\begin{equation} \label{quasid}
\begin{dcases}
\frac{1}{2} \, |\xi|^2 \leq \big\langle D\S(y) \, \xi, \xi \big\rangle \leq 2 \, |\xi|^2, \\
\frac{1}{2}\,|\xi|^2 \leq \big\langle (D\T   \circ\S)(y) \, \xi, \xi \big\rangle \leq 2\,|\xi|^2, \\
\frac{1}{2} \, |\xi| \leq \big|  D\S(y) \, \xi \big|\leq 2 \,|\xi|,
\end{dcases}
\qquad \text{for all } y \in B_{r_0}, \, \xi \in \R^n,
\end{equation}
for some~$r_0 \in (0, 1)$ suitably small, in dependence of~$n$,~$\alpha$, and~$\| h \|_{C^{1, \alpha}(\R^{n - 1})}$ only. Moreover, by differentiating~\eqref{matrS} and~\eqref{matrT}, taking advantage of estimate~\eqref{stime:H1}, and recalling definition~\eqref{eta}, we find
\begin{equation}\label{bDDST}
\big| {D^2\S(y)} \big| + \big| {D_y (D\T\circ\S)(y)} \big| \leq C\,y_n^{\a-1}\quad\text{for all }y\in \R^n_+.
\end{equation}

We now transform the operator~$Q$ via~$\S$. Recalling~\eqref{Omegasupgraph},~\eqref{Psimapprop}, and since~$\S(0)=0$, thanks to~\eqref{bound:DST} we can find~$\tau \in (0, 1)$, depending only on~$n$,~$\alpha$, and~$\| h \|_{C^{1,\a}(\R^{n - 1})}$, such that
\begin{equation}\label{PsiCinB}
\S \big( {\mathcal{C}_{4r}^+} \big)\subset \Omega\cap B_R
\end{equation}
for all~$r\in (0,\tau R)$. Let~$r \in (0, \frac{r_0}{8})$ be as such and define~$\tilde{u} \coloneqq u \circ \S$. As~$u$ is a weak su\-per\-so\-lu\-tion of~$Q u = 0$ in~$\Omega$, simple computations show that~$\tilde{u} \in W^{1,p}(\mathcal{C}_{2r}^+)\cap \W^{s,q}(\mathcal{C}_{2r}^+)\cap C^0\big(\overline{\mathcal{C}^+_{2r}}\big)$ is a weak supersolution of~$\widetilde{Q} \tilde{u} = 0$ in~$\mathcal{C}^+_{2 r}$, where~$\widetilde{Q}$ is defined by~$\widetilde{Q} \coloneqq \widetilde{Q}_{\text{L}} + \widetilde{Q}_{\text{N}}$ and 
\begin{align*}
\widetilde{Q}_{\text{L}} \, \tilde{u}(y) & \coloneqq - \dive \widetilde{A}(y, D\tilde{u}(y)), \quad \mbox{with } \widetilde{A}(y, z) \coloneqq \mathfrak{c}(y) A \big( {\S(y), z  (D\T \circ \S)(y)} \big) (D\T \circ \S)(y)^T, \\
\widetilde{Q}_{\text{N}} \, \tilde{u}(y) & \coloneqq 2 \, \mathfrak{c}(y) \, \PV \int_{\R^n} \phi \big( {\tilde{u}(y) - \tilde{u}(z)} \big) \frac{B \big( {\S(y), \S(z)} \big)}{|\S(y) - \S(z)|^{n + s q}} \, \mathfrak{c}(z) \, dz,
\end{align*}
for all~$y \in \mathcal{C}^+_{2 r}$.

\subsection*{Step 2: Definition of a subsolution~\texorpdfstring{$v$}{} for~\texorpdfstring{$\widetilde{Q}_L$}{}}

To establish~\eqref{normderneg}, we need to construct a suitable subsolution. Let~$\beta \in (0, 1)$,~$\delta \in \left( 0, \frac{1}{4} \right]$, and define
$$
\varphi(y) \coloneqq - \frac{\delta}{r^2} \, |y'|^2 + \frac{y_n}{2 r} + \frac{y_n^{1 + \beta}}{2 r^{1 + \beta}} \quad \mbox{for } y \in \mathcal{C}^+_{2 r, r}.
$$
Note that~$\varphi \in C^\infty(\mathcal{C}^+_{2 r, r})$ and
$$
D \varphi(y) = \left( - \frac{2 \delta}{r^2} \, y', \frac{1}{2 r} + \frac{1 + \beta}{2 r^{1 + \beta}} \, y_n^\beta \right) \quad \mbox{for all } y \in \mathcal{C}^+_{2 r, r},
$$
so that, in particular,~$D \varphi \ne 0$ in~$\mathcal{C}^+_{2 r, r}$. Also, the matrix~$D^2\vphi$ is diagonal and
\begin{equation}\label{diagonal}
    \partial_{y_i y_i}^2 \vphi(y)=-\frac{2 \d}{r^2}\quad\text{and}\quad \partial_{y_n y_n}^2 \vphi(y)=\frac{\b(1+\b)}{2\,r^{1+\b}} \,y_n^{\b-1},
\end{equation}
for every~$i = 1, \ldots, n - 1$. For~$\varepsilon \in (0, 1]$ to be chosen later, we set~$v = v_\varepsilon \coloneqq \varepsilon \varphi$.

We claim that, if~$\beta \in (0, \alpha)$ and~$\delta$ is small enough, in dependence of~$n$,~$p$,~$\L$,~$\alpha$,~$\beta$, and~$\S$ only, it holds
\begin{equation} \label{QLclaim}
\widetilde{Q}_L v_\varepsilon(y) \le 0 \quad \mbox{for all } y \in \mathcal{C}^+_{r, \delta r} \mbox{ and } \varepsilon \in (0, 1].
\end{equation}
To verify this, we first observe that, by exploiting the structural assumptions~\eqref{ass:A} and~\eqref{ass:A2}, together with~\eqref{detundercontrol},~\eqref{bound:DST},~\eqref{quasid}, and~\eqref{bDDST}, the function~$\widetilde{A}$ satisfies
\begin{equation}\label{new:coercive}
    \begin{dcases}
        \big|\partial_y \widetilde{A}(y,z)\big|\leq C\,\big(  |z|^2+\mu^2\big)^{\frac{p-2}{2}}\,|z| \, y_n^{\a-1}
        \\
        \big|\partial_z \widetilde{A}(y,z)|\leq C\,\big(  |z|^2+\mu^2\big)^{\frac{p-2}{2}}
        \\
        \<\partial_z \widetilde{A}(y,z)\xi,\xi\>\geq C^{-1} \,\big(  |z|^2+\mu^2\big)^{\frac{p-2}{2}}\,|\xi|^2
    \end{dcases}
 \quad\, \text{for all } y\in \mathcal{C}^+_{2r}, \, z \in\R^n\setminus\{0\}, \, \xi\in \R^n.
\end{equation}
Within this step,~$C$ depends only on~$n$,~$p$,~$\Lambda$,~$\alpha$,~$\beta$, and~$\S$.
Since~$Dv \neq 0$,~$D^2 v$ is diagonal, and~$\widetilde A\in C^1\big(\mathcal{C}^+_{2r}\times (\R^n\setminus\{0\})\big)$, the chain rule entails
$$
\widetilde{Q}_L v = - \sum_{i=1}^n \partial_{y_i} \widetilde {A}^i(y,Dv) - \sum_{i=1}^{n}\partial_{z_i}\widetilde{A}^i(y,Dv)\,\partial_{y_i y_i}v \quad \mbox{in } \mathcal{C}^+_{2 r, r}.
$$
Therefore, by using~\eqref{diagonal},~\eqref{new:coercive}, and the fact that
\begin{equation*}
    \frac{\varepsilon}{2 r}\leq |Dv|\leq \frac{C \, \varepsilon}{r} \quad \mbox{in } \mathcal{C}^+_{2 r, r}, 
\end{equation*}
we obtain
\begin{align*}
\widetilde{Q}_L v(y) & \le - \varepsilon \, \frac{y_n^{\beta - 1}}{r^{1 + \beta}} \left( \frac{\varepsilon^2}{r^2} + \mu^2 \right)^{\frac{p - 2}{2}} \left\{ \frac{\beta(1 + \beta)}{C} - C \, \delta \left( \frac{y_n}{r} \right)^{1 - \beta} - C \, r^{1 + \beta} y_n^{\alpha - \beta} \right\} \\
& \le - \varepsilon \, \frac{y_n^{\beta - 1}}{r^{1 + \beta}} \left( \frac{\varepsilon^2}{r^2} + \mu^2 \right)^{\frac{p - 2}{2}} \left\{ \frac{1}{C} - C \, \delta^{2 - \beta} - C \, r^{1 + \alpha} \delta^{\alpha - \beta} \right\} \quad \mbox{for all } y \in \mathcal{C}^+_{r, \delta r},
\end{align*}
From this, claim~\eqref{QLclaim} immediately follows by taking~$\delta$ sufficiently small.

\subsection*{Step 3: Extending~\texorpdfstring{$v$}{} to a subsolution for~\texorpdfstring{$\widetilde{Q}_N$}{}}

Next, we extend~$v$ to a bounded function~$\tilde{v}$ defined on the whole~$\R^n$ satisfying
\begin{equation} \label{QNclaim}
\widetilde{Q}_N \tilde{v}(y) \le 0 \quad \mbox{for all } y \in \mathcal{C}^+_{r, \delta r},
\end{equation}
provided~$\delta$ is sufficiently small. We stress that the nonlocal operator~$\widetilde{Q}_N \tilde{v}$ is well-defined in~$\mathcal{C}_{r, \delta r}^+$ in the pointwise sense, as~$\tilde{v}$ is globally bounded and smooth inside~$\mathcal{C}_{r, \delta r}^+$ with non-vanishing gradient---this can be easily justified through the computations made, for instance, in~\cite[Section~3]{KKL19}.

In order to achieve this, we let~$\widetilde{\varphi}$ be any bounded, Lipschitz continuous, and compactly supported extension of~$\varphi$ to~$\R^n$ satisfying 
\begin{equation} \label{phitildecond}
\begin{dcases}
\widetilde{\varphi}(y) = - \frac{\delta}{r^2} |y'|^2 & \quad \mbox{for all } y \in B_{2 r}' \times (- r, 0], \vphantom{y \in B_{\frac{r}{4}} \!  \left(\frac{3 r}{2} \, e_n \right)} \\
\widetilde{\varphi}(y) = M & \quad \mbox{for all } y \in B_{\frac{r}{4}} \!  \left(\frac{3 r}{2} \, e_n \right), \\
\widetilde{\varphi}(y) \le 0 & \quad \mbox{for all } y \in \R^n \setminus \Big( {\mathcal{C}^+_{r, \delta r} \cup \big( {B'_{3 r} \times [\delta r, 2r)} \big)} \Big), \vphantom{y \in B_{\frac{r}{4}} \!  \left(\frac{3 r}{2} \, e_n \right)} \\
-2 \le \widetilde{\varphi}(y) \le M & \quad \mbox{for all } y \in \R^n, \vphantom{y \in B_{\frac{r}{4}} \!  \left(\frac{3 r}{2} \, e_n \right)}
\end{dcases}
\end{equation}
for some~$M \ge 2$ to be chosen suitably large. As before, we also set~$\tilde{v} = \tilde{v}_\varepsilon \coloneqq \varepsilon \widetilde{\varphi}$.

We write
\begin{equation} \label{QNdeco}
\widetilde{Q}_N \tilde{v}(y) = 2 \, \mathfrak{c}(y) \Big( {I(y) + E_1(y) + E_2(y)} \Big),
\end{equation}
where
\begin{align*}
I(y) & \coloneqq \PV \int_{B_{\frac{r}{2}}(y)} \phi \big( {\tilde{v}(y) - \tilde{v}(z)} \big) \frac{B \big( {\S(y), \S(z)} \big)}{|\S(y) - \S(z)|^{n + s q}} \, \mathfrak{c}(z) \, dz, \\
E_1(y) & \coloneqq \int_{B_{\frac{r}{4}} \! \left(\frac{3 r}{2} e_n \right)} \phi \big( {\tilde{v}(y) - \tilde{v}(z)} \big) \frac{B \big( {\S(y), \S(z)} \big)}{|\S(y) - \S(z)|^{n + s q}} \, \mathfrak{c}(z) \, dz, \\
E_2(y) & \coloneqq \int_{\R^n \setminus \left( B_{\frac{r}{2}}(y) \cup B_{\frac{r}{4}} \!  \left(\frac{3 r}{2} e_n \right) \right)} \phi \big( {\tilde{v}(y) - \tilde{v}(z)} \big) \frac{B \big( {\S(y), \S(z)} \big)}{|\S(y) - \S(z)|^{n + s q}} \, \mathfrak{c}(z) \, dz.
\end{align*}


By using that~$\widetilde{\varphi} = \varphi \le 1$ in~$\mathcal{C}^+_{r, \delta r}$,~$\widetilde{\varphi} = M$ in~$B_{\frac{r}{4}} \!  \left(\frac{3 r}{2} \, e_n \right)$, and~$\widetilde{\varphi}\geq -2$ in~$\R^n$, in combination with the monotonicity of~$\phi$, bounds~\eqref{detundercontrol} and~\eqref{SLip}, as well as assumptions~\eqref{ass:B}-\eqref{ass:phi}, we obtain
\begin{equation} \label{E1est}
\begin{aligned}
E_1(y) & \le - \frac{\varepsilon^{q - 1} (M - 1)^{q - 1}}{\L^2} \int_{B_{\frac{r}{4}} \!  \left(\frac{3 r}{2} e_n \right)} \frac{\mathfrak{c}(z)}{|\S(y) - \S(z)|^{n + s q}} \, dz \\
& \le - \frac{\varepsilon^{q- 1} M^{q - 1}}{C \, r^{s q}} \quad \mbox{for all } y \in \mathcal{C}^+_{r, \delta r}
\end{aligned}
\end{equation}
and
\begin{equation} \label{E2est}
\begin{aligned}
E_2(y) & \le \L^2 3^{q - 1} \varepsilon^{q - 1} \int_{\R^n \setminus \left( B_{\frac{r}{2}}(y) \cup B_{\frac{r}{4}} \!  \left(\frac{3 r}{2} e_n \right) \right)} \frac{\mathfrak{c}(z)}{|\S(y) - \S(z)|^{n + s q}} \, dz \\
& \leq \frac{C\,\varepsilon^{q-1}}{r^{sq}} \quad \mbox{for all } y \in \mathcal{C}^+_{r, \delta r}.
\end{aligned}
\end{equation}
Here,~$C$ is a constant depending only on~$n$,~$q$,~$s$,~$\Lambda$,~$\alpha$,~$\beta$, and~$\S$.

We now inspect the term~$I$. We write
\begin{equation} \label{useful0}
\frac{B\big(\S(y), \S(z)\big)\, \mathfrak{c}(z)}{|\S(y)-\S(z)|^{n+sq}} = \frac{B\big(\S(y),\S(y)\big)\,\mathfrak{c}(y)}{|D\S(y)\,(y-z)|^{n+sq}} + \mathcal{R}_1(y,z)+\mathcal{R}_2(y,z)+\mathcal{R}_3(y,z),
\end{equation}
with
\begin{align*}
\mathcal{R}_1(y,z) & \coloneqq B\big({\S(y),\S(y)}\big)\,\mathfrak{c}(y)\,\bigg( {\frac{1}{|\S(y)-\S(z)|^{n+sq}}-\frac{1}{|D\S(y)\,(y-z)|^{n+sq}}} \bigg), \\
\mathcal{R}_2(y,z) & \coloneqq B\big( {\S(y),\S(y)} \big) \, \frac{\mathfrak{c}(z)-\mathfrak{c}(y)}{|\S(y)-\S(z)|^{n+sq}}, \\
\mathcal{R}_3(y,z) & \coloneqq \mathfrak{c}(z) \, \frac{B\big( {\S(y),\S(z)} \big)-B\big( {\S(y),\S(y)} \big)}{|  \S(y)-\S(z)|^{n+sq}}.
\end{align*}
We claim that, for~$i=1,2,3$ and for all~$y\in \mathcal{C}^+_{r, \delta r}$, it holds
\begin{equation}\label{useful1}
|\mathcal{R}_i(y,z)|\leq C \begin{cases}
    |y-z|^{-n-sq+\a} & \quad\text{for all } z\in B_{\frac{r}{2}}(y),
    \\
    y_n^{\a-1}|y-z|^{-n-sq+1} & \quad \text{for all } z\in B_{\frac{y_n}{2}}(y).
\end{cases}
\end{equation}
By taking advantage of~\eqref{ass:B},~\eqref{ass:B2},~\eqref{detundercontrol},~\eqref{detundercontrol1}, and~\eqref{SLip}, we immediately deduce the validity of~\eqref{useful1} for~$i = 2$ as well as the following stronger inequality for~$i = 3$:
$$
|\mathcal{R}_3(y,z)| \le C \, |y - z|^{- n - s q + 1} \quad \mbox{for all } y \in \mathcal{C}^+_{r, \delta r}, \, z \in B_{\frac{r}{2}}(y).
$$
On the other hand, using~\eqref{ass:B},~\eqref{detundercontrol},~\eqref{SLip}, and~\eqref{quasid}, together with the numerical inequality~$\big| {A^P - B^P} \big| \le P (A + B)^{P - 1} |A - B|$, valid for every~$P > 1$ and~$A, B \ge 0$, we find
\begin{align*}
|\mathcal{R}_1(y,z)| & \le C \, \frac{\Big| {|D\S(y)(z-y)|^{n+sq} - |\S(z)-\S(y)|^{n+sq}} \Big|}{|\S(y)-\S(z)|^{n+sq}|D\S(y)(z-y)|^{n+sq}} \\
& \le C \, \frac{\big| {\S(z) - \S(y) - D\S(y)(z - y)} \big|}{|y - z|^{n + s q + 1}},
\end{align*}
from which~\eqref{useful1} for~$i = 1$ follows at once by noticing that
\begin{align*}
\big| {\S(z) - \S(y) - D\S(y)(z - y)} \big| & \le |y - z| \int_0^1 \big| {D\S(t z + (1 - t) y) - D\S(y)} \big| \, dt \\
& \le C \begin{cases}
|y - z|^{1 + \alpha} & \quad \mbox{for all } z \in B_{\frac{r}{2}}(y), \\
y_n^{\alpha - 1} |y - z|^2 & \quad \mbox{for all } z \in B_{\frac{y_n}{2}}(y),
\end{cases}
\end{align*}
thanks to~\eqref{bound:DST},~\eqref{bDDST}, and the fact that the segment joining~$y$ and~$t z + (1 - t) y$ lies in the half-space~$\left\{ w \in \R^n : w_n \ge \frac{y_n}{2} \right\}$ for every~$t \in [0, 1]$.

Observe now that
$$
y_n^{\alpha - 1} \int_{B_{\frac{y_n}{2}}(y)} \frac{dz}{|y - z|^{n + s q - q}} + \int_{B_{\frac{r}{2}}(y) \setminus B_{\frac{y_n}{2}}(y)} \frac{dz}{|y - z|^{n + s q - \alpha - q + 1}} \le C \, r^{(1 - s) q + \alpha - 1} L_\alpha \! \left( \frac{y_n}{r} \right),
$$
where, for~$t \in (0, 1)$ and~$\gamma \in \R$, we set
$$
L_\gamma(t) \coloneqq \begin{cases}
t^{(1 - s) q + \gamma - 1} & \quad \mbox{if } (1 - s) q + \gamma < 1, \\
- \log t & \quad \mbox{if } (1 - s) q + \gamma = 1, \\
1 & \quad \mbox{if } (1 - s) q + \gamma > 1.
\end{cases}
$$
Also, by means of~\eqref{ass:phi}, of the fundamental theorem of calculus, and of the estimate
$$
|D\tilde{v}(y)|\leq \frac{C\,\varepsilon}{r}\quad\text{for all } y \in B'_{2r} \times (-r,r),
$$
we see that
$$
\left| \phi \big( {\tilde{v}(y) - \tilde{v}(z)} \big) \right| \le C \, \frac{\varepsilon^{q - 1}}{r^{q - 1}} \, |y - z|^{q - 1} \quad \mbox{for all } y \in \mathcal{C}^+_{r, \delta r}, \, z \in B_{\frac{r}{2}}(y).
$$
In light of these facts,~\eqref{useful0},~\eqref{useful1}, and recalling the definition of~$I$, we have that
$$
I(y) = \mathfrak{c}(y) B \big( {\S(y), \S(y)} \big) \, \PV \int_{B_{\frac{r}{2}}(y)} \frac{\phi \big( {\tilde{v}(y) - \tilde{v}(z)} \big)}{| D\S(y)(y-z)|^{n + s q}} \, dz + \mathcal{E}(y),
$$
with
\begin{equation} \label{Errest}
|\mathcal{E}(y)| \le C \, \varepsilon^{q - 1} r^{- s q + \alpha} L_\alpha \! \left( \frac{y_n}{r} \right) \quad \mbox{for every } y \in \mathcal{C}^+_{r, \delta r}.
\end{equation}
Since, by symmetry,
$$
\PV \int_{B_{\frac{r}{2}}(y)} \frac{\phi \big( {D \tilde{v}(y) \cdot (y - z)} \big)}{|D\S(y)(y-z)|^{n + s q}} \, dz = 0,
$$
the previous identity can be rewritten as
\begin{equation} \label{I=I1+err}
I(y) = \mathfrak{c}(y) B \big( {\S(y), \S(y)} \big) I_1(y) + \mathcal{E}(y) \quad \mbox{for every } y \in \mathcal{C}^+_{r, \delta r},
\end{equation}
with~$\mathcal{E}$ satisfying~\eqref{Errest} and
$$
I_1(y) \coloneqq \int_{B_{\frac{r}{2}}(y)} \frac{\phi \big( {\tilde{v}(y) - \tilde{v}(z)} \big) - \phi \big( {D \tilde{v}(y) \cdot (y - z)} \big)}{|D\S(y)(y-z)|^{n + s q}} \, dz.
$$

We now claim that
\begin{equation} \label{claimI1}
I_1(y) \le - \frac{1}{C} \frac{\varepsilon^{q - 1} L_0 \! \left( \frac{y_n}{r} \right)}{r^{s q}} \quad \mbox{for all } y \in \mathcal{C}^+_{r, \delta r},
\end{equation}
for some constant~$C \ge 1$, provided~$\delta$ is small enough, all in dependence of~$n$,~$q$,~$s$,~$\Lambda$, and~$\beta$ only.

The remaining of Step~3 is essentially occupied by the proof of this claim. First, we apply the change of variables~$\ell \coloneqq \frac{z - y}{y_n}$ and observe that~\eqref{claimI1} is equivalent to showing that
\begin{equation} \label{hopfclaim1}
\begin{aligned}
& \int_{B_{\frac{1}{2 x_n}}} \Bigg\{ {\widetilde{\phi} \, \bigg( {- 2 \delta \Big( {2 x' \cdot \ell' + x_n |\ell'|^2} \Big) + (1 + \ell_n)_+ - 1 + x_n^{\beta} \left( (1 + \ell_n)_+^{1 + \beta} - 1 \right)} \bigg)} \\
& \hspace{80pt} - {\widetilde{\phi} \, \bigg( {\left( - 4 \delta x', 1 + (1 + \beta) x_n^\beta \right) \cdot \ell} \bigg)} \Bigg\} \frac{d\ell}{|D\S(r x)\,\ell|^{n + s q}} \ge \frac{1}{C} \frac{L_0(x_n)}{x_n^{(1 - s) q - 1}},
\end{aligned}
\end{equation}
for all~$x \coloneqq \frac{y}{r} \in \mathcal{C}^+_{1, \delta}$, and where~$\widetilde{\phi} \coloneqq (\varepsilon x_n)^{1 - q} \, \phi \! \left( \frac{\varepsilon x_n}{2} \, \cdot \, \right)$. We make a further substitution and consider the new variables~$w$ defined by~$\ell = (w', d \cdot w)$, with
\begin{equation}\label{deffd}
d' \coloneqq \frac{4 \delta x'}{1 + (1 + \beta) x_n^\beta} \quad \mbox{and} \quad d_n \coloneqq \frac{1}{1 + (1 + \beta) x_n^\beta}.
\end{equation}
In particular, it defines a bi-Lipschitz map on~$\R^n$ such that
\begin{equation}\label{ellw}
\frac{1}{2}\,|w|\leq \big| {(w',d\cdot w)} \big| \leq 2\,|w|\quad\text{for all } w\in \R^n,
\end{equation}
provided~$\delta$ is sufficiently small. Inequality~\eqref{hopfclaim1} then becomes
\begin{equation} \label{Hopfclaimtechequiv}
\int_{\R^n} \mathcal{F}(w) \, dw \ge \frac{1}{C} \frac{L_0(x_n)}{d_n x_n^{(1 - s) q - 1}},
\end{equation}
where
\begin{align*}
\mathcal{F}(w) & \coloneqq \Bigg\{ {\widetilde{\phi} \, \bigg( {- 2 \delta \Big( {2 x' \cdot w' + x_n |w'|^2} \Big) + (1 + d \cdot w)_+ - 1 + x_n^{\beta} \left( (1 + d \cdot w)_+^{1 + \beta} - 1 \right)} \bigg)} \\
& \hspace{215pt} - {\widetilde{\phi}(w_n)} \Bigg\} \frac{\chi_{[0, 1)} \! \left( 2 x_n \sqrt{|w'|^2 + (d \cdot w)^2} \right)}{| D\S(r x)\,(w', d \cdot w)|^{n + s q}}.
\end{align*}
We now look for a lower bound on~$\mathcal{F}$. First, let~$w \in B_{\frac{1}{2}}$. In this case, we observe that
$$
|d \cdot w| = \left| \frac{4 \delta x' \cdot w' + w_n}{1 + (1 + \beta) x_n^\beta} \right| \le \frac{4 \delta |w'| + |w_n|}{1 + (1 + \beta) x_n^\beta} \le (1 + 4 \delta) |w| \le \frac{3}{4},
$$
if we take~$\delta \in \left( 0, \frac{1}{8} \right]$. Hence, writing
$$
N_q(w) \coloneqq \begin{cases}
|w|^{q - 2} & \quad \mbox{if } q \ge 2, \\
|w_n|^{q - 2} & \quad \mbox{if } q \in (1, 2),
\end{cases}
$$
exploiting the monotonicity of~$\phi$, inequalities~\eqref{quasid},~\eqref{ellw}, and
$$
\left| (1+d\cdot w)^{1+\b}-1 \right| \leq 4 \, |d\cdot w| \leq 6 \, |w| \quad \mbox{for every } w \in B_{\frac{1}{2}},
$$
as well as the bound
\begin{equation} \label{elemineab}
|\widetilde{\phi}(a) - \widetilde{\phi}(b)| \le C_q \Lambda \big( {|a| + |b|} \big)^{q - 2} |a - b| \quad \mbox{for all } (a, b) \in \R^2 \setminus \{ (0, 0) \},
\end{equation}
which holds for some constant~$C_q > 0$ depending only on~$q$ thanks to the fact that~$\phi$ fulfils assumption~\eqref{ass:phi2}---see, e.g.,~\cite[Lemma~2.1]{dam}---, and ultimately recalling~\eqref{deffd} and the fact that~$x\in \mathcal{C}^+_{1,\d}$, we estimate
\begin{align*}
|\mathcal{F}(w)| & \le C \, \frac{N_q(w)}{|w|^{n + s q}} \left| - 2 \delta x_n |w'|^2 + x_n^{\beta} \left( (1 + d \cdot w)^{1 + \beta} - 1 - (1 + \beta) d \cdot w \right) \right| \\
& \le C \, \frac{N_q(w)}{|w|^{n + s q}} \, \left( \delta^2 |w'|^2 + x_n^{\beta} \left( d \cdot w \right)^{2} \right) \le C \delta^\beta \frac{N_q(w)}{|w|^{n - 2 + s q}} \qquad \mbox{for a.e.~} w \in B_{\frac{1}{2}}.
\end{align*}
On the other hand, the monotonicity of~$\phi$ and again~\eqref{quasid},~\eqref{ellw}, and~\eqref{elemineab} lead us to
\begin{align*}
\mathcal{F}(w) & \ge \frac{\widetilde{\phi} \left( - 5 \delta |w'| + d \cdot w - x_n^{\beta} \right) - \widetilde{\phi}(w_n)}{\big| { D\S(r x)\,(w', d \cdot w)} \big|^{n + s q}} \, \chi_{[0, 1)} \! \left( 2 x_n \sqrt{|w'|^2 + (d \cdot w)^2} \right) \\
& \ge - C \, \frac{N_q(w)}{|w|^{n + s q}} \left( \delta |w| + x_n^\beta + |d \cdot w - w_n| \right) \chi_{B_{\frac{1}{x_n}}}(w) \\
& \ge - C \delta^\beta \frac{N_q(w)}{|w|^{n - 1 + s q}} \, \chi_{B_{\frac{1}{x_n}}}(w) \qquad \mbox{for a.e.~} w \in \R^n \setminus B_{\frac{1}{2}}.
\end{align*}
These two estimates yield that
\begin{equation} \label{festinwnge-1}
\int_{\{ w_n \ge - 10 \}} \mathcal{F}(w) \, dw \ge - C \delta^\beta \int_{B_{\frac{1}{x_n}}} \frac{N_q(w) \min \{ |w|, 1 \}}{|w|^{n - 1 + s q}} \, dw \ge - C \delta^\beta \frac{L_0(x_n)}{x_n^{(1 - s) q - 1}}.
\end{equation}
The last inequality is straightforward if~$q \ge 2$. When~$q \in (1, 2)$, it is a consequence of the following calculation, valid for every~$R \ge 2$:
\begin{align*}
& \int_{B_R} \frac{|w_n|^{q - 2} \min \{ |w|, 1 \}}{|w|^{n - 1 + sq}} \, dw \le C \int_0^R \Bigg( {\int_0^R \, \frac{\min \left\{ \sqrt{\rho^2 + w_n^2}, 1 \right\}}{\left( \rho^2 + w_n^2 \right)^{\frac{n - 1 + s q}{2}}} \, \rho^{n - 2} \, d\rho} \Bigg) \frac{dw_n}{w_n^{2 - q}} \\
& \hspace{40pt} \le C \int_0^{+\infty} \Bigg( {\sqrt{1 + t^2} \int_0^{\frac{1}{\sqrt{1 + t^2}}} \, \frac{dw_n}{w_n^{1 - (1 - s) q}} + \int_{\frac{1}{\sqrt{1 + t^2}}}^{\frac{R}{t}} \, \frac{dw_n}{w_n^{2 - (1 - s) q}}} \Bigg) \frac{t^{n - 2}}{\left( 1 + t^2 \right)^{\frac{n - 1 + s q}{2}}} \, dt \\
& \hspace{40pt} \le C \left\{ \int_0^{+\infty} \frac{t^{n - 2}}{\left( 1 + t^2 \right)^{\frac{n + q - 2}{2}}} \, dt + \int_0^{+\infty} \Bigg( {\int_{\frac{1}{\sqrt{1 + t^2}}}^{\frac{R}{t}} \, \frac{dw_n}{w_n^{2 - (1 - s) q}}} \Bigg) \frac{t^{n - 2}}{\left( 1 + t^2 \right)^{\frac{n - 1 + s q}{2}}} \, dt \right\} \\
& \hspace{40pt} \le C R^{(1 - s) q - 1} L_0 \! \left( \frac{1}{R} \right).
\end{align*}
In light of~\eqref{festinwnge-1}, we are left with bounding the integral over~$\{ w_n < - 10 \}$. Using that~$\phi$ is odd and monotone, along with the bounds~\eqref{quasid} and~\eqref{ellw}, for~$w \in \{ w_n < - 10 \} \cap B_{1/\delta}$ we have 
\begin{align*}
\mathcal{F}(w) & \ge \frac{\widetilde{\phi}(- w_n) - \widetilde{\phi} \left( 5 \delta |w| + 2 \right)}{| D\S(r x)\,(w', d \cdot w)|^{n + s q}} \, \chi_{[0, 1)} \! \left( 2 x_n \sqrt{|w'|^2 + (d \cdot w)^2} \right) \\
& \ge \frac{1}{C} \frac{\widetilde{\phi}(- w_n) - \widetilde{\phi}(10)}{|w|^{n + s q}} \, \chi_{B_{\frac{1}{4 x_n}}}(w),
\end{align*}
while, using also that~$\phi$ satisfies the growth assumption~\eqref{ass:phi}, for~$w \in \{ w_n < - 10 \} \setminus B_{1/\delta}$ it holds
\begin{align*}
\mathcal{F}(w) & \ge \frac{\widetilde{\phi}(- w_n) - \widetilde{\phi}(10) + \widetilde{\phi}(10) - \widetilde{\phi} \left( 5 \delta |w| + 1 + x_n^{\beta} \right)}{| D\S(r x)\,(w', d \cdot w)|^{n + s q}} \, \chi_{[0, 1)} \! \left( 2 x_n \sqrt{|w'|^2 + (d \cdot w)^2} \right) \\
& \ge \frac{1}{C} \frac{\widetilde{\phi}(- w_n) - \widetilde{\phi}(10)}{|w|^{n + s q}} \, \chi_{B_{\frac{1}{4 x_n}}}(w) - \frac{C \delta^{q - 1}}{|w|^{n + 1 - (1 - s) q}} \, \chi_{B_{\frac{1}{x_n}}}(w).
\end{align*}
Putting these two estimates together and using the fact that, thanks to assumption~\eqref{ass:phi2},
$$
\widetilde{\phi}(- w_n) - \widetilde{\phi}(10) \ge \frac{2^{1 - q}}{(q - 1) \Lambda} \Big( {(-w_n)^{q - 1} - 10^{q - 1}} \Big) \quad \mbox{for every } w_n < -10,
$$
we compute, after a change of coordinates,
\begin{equation} \label{Hopftech2}
\begin{aligned}
& \int_{\{ w_n < - 10 \}} \mathcal{F}(w) \, dw \\
& \hspace{30pt} \ge \frac{1}{C} \int_{B_{\frac{1}{4 x_n}} \cap \{ w_n < -10\}} \frac{(-w_n)^{q - 1} - 10^{q - 1}}{|w|^{n + s q}} \, dw - C \delta^{q - 1} \int_{B_{\frac{1}{x_n}} \setminus B_{\frac{1}{\delta}}} \frac{dw}{|w|^{n + 1 - (1 - s) q}} \\
& \hspace{30pt} \ge \frac{10^{(1 - s) q - 1}}{C} \int_{B_{\frac{1}{40 x_n}} \cap \{ z > 1 \}} \frac{z_n^{q - 1} - 1}{|z|^{n + s q}} \, dz - C \delta^{q - 1} \frac{L_0(x_n)}{x_n^{(1 - s) q - 1}}.
\end{aligned}
\end{equation}
Through a further changing variables, we see that
\begin{align*}
\int_{B_{\frac{1}{40 x_n}} \cap \{ z > 1 \}} \frac{z_n^{q - 1} - 1}{|z|^{n + s q}} \, dz & \ge \int_1^{\frac{1}{80 x_n}} \Bigg( {\int_{B'_{z_n}} \frac{dz'}{\big( {|z'|^2 + z_n^2 } \big)^{\frac{n + s q}{2}}}} \Bigg) \big( {z_n^{q - 1} - 1} \big) \, dz_n \\
& = \Haus^{n - 2}(\partial B_1') \Bigg( {\int_0^1 \frac{t^{n - 2}}{\big( {1 + t^2} \big)^{\frac{n + s q}{2}}} \, d\rho} \Bigg) \Bigg( {\int_1^{\frac{1}{80 x_n}} \frac{z_n^{q - 1} - 1}{z_n^{1 + s q}} \, dz_n} \Bigg) \\
& \ge \frac{1}{C} \frac{L_0(x_n)}{x_n^{(1 - s) q - 1}}.
\end{align*}
From this,~\eqref{festinwnge-1}, and~\eqref{Hopftech2} it follows that inequality~\eqref{Hopfclaimtechequiv} holds for some constant~$C \ge 1$, provided~$\delta$ is sufficiently small, all in dependence of~$n$,~$q$,~$s$,~$\Lambda$, and~$\beta$ only. Consequently, claim~\eqref{claimI1} is proved.

From~\eqref{ass:B},~\eqref{detundercontrol},~\eqref{QNdeco},~\eqref{E1est},~\eqref{E2est},~\eqref{Errest},~\eqref{I=I1+err}, and~\eqref{claimI1} we immediately infer that
\begin{equation*}
    \widetilde{Q}_N \tilde{v}(y)\leq -\frac{\varepsilon^{q-1}}{r^{sq}}\bigg\{ \frac{M^{q-1} }{C} + \frac{1}{C} \, L_0\Big( {\frac{y_n}{r}} \Big)-C-C\,r^\a\,L_\a\Big(\frac{y_n}{r}\Big)  \bigg\} \quad \mbox{for all } y\in \mathcal{C}^+_{r, \delta r},
\end{equation*}
whence~\eqref{QNclaim} holds true, provided we take~$\delta$ sufficiently small (when~$(1 - s) q < 1$) or~$M$ sufficiently large (when~$(1 - s) q \ge 1$).

\subsection*{Step 4: Conclusion}
Taking advantage of~\eqref{QLclaim},~\eqref{QNclaim}, and of the locality of~$\widetilde{Q}_L$, we see that
$$
\widetilde{Q} \tilde{v}_\varepsilon = \widetilde{Q}_L v_\varepsilon + \widetilde{Q}_N \tilde{v}_\varepsilon \le 0 \quad \mbox{in } \mathcal{C}^+_{r, \delta r},
$$
for every~$\varepsilon \in (0, 1)$.

We now show that, by taking~$\varepsilon$ tiny enough, we can make~$\tilde{v}_\varepsilon = \varepsilon \widetilde{\varphi}$ smaller than~$\tilde{u}$ in the whole of~$\R^n$. Indeed, by~\eqref{phitildecond} we have that~$\tilde{v}_\varepsilon \le 0$ in~$\R^n \setminus \big( {\mathcal{C}^+_{r, \delta r} \cup \left( B_{3 r}' \times [\delta r, 2 r) \right)} \big)$ and that~$\tilde{v}_\varepsilon \le \varepsilon M$ in~$B_{3 r}' \times [\delta r, 2 r)$. As~$\S \big( {B_{3 r}' \times [\delta r, 2 r)} \big) \subset \subset \Omega$, thanks to~\eqref{Omegasupgraph},~\eqref{Psimapprop},~\eqref{PsiCinB}, and since~$u$ is positive and continuous in~$\Omega$, then~$m \coloneqq \inf_{B_{3 r}' \times [\delta r, 2 r)} \tilde{u} > 0$. By choosing~$\varepsilon \le \frac{m}{M}$ and recalling that~$\tilde{u}$ is non-negative in the whole of~$\R^n$, we infer that~$\tilde{v}_\varepsilon \le \tilde{u}$ in~$\R^n \setminus \mathcal{C}^+_{r, \delta r}$.

Thanks to the weak comparison principle of Proposition~\ref{WCPprop}, we then conclude that~$\tilde{v}_\varepsilon \le \tilde{u}$ in~$\R^n$. This yields in particular that
$$
\tilde{u}(0', y_n) \ge \varepsilon\, \varphi(0', y_n) \ge \frac{\varepsilon}{2r} \, y_n \quad \mbox{for all } y_n \in (0, r). 
$$
Recalling that~$D\T(0)=\mathrm{Id}_n$, by rephrasing this inequality in terms of the original variable~$x$ and of the function~$u$, we are easily led to~\eqref{normderneg}. The proof is thus complete.

\section{A strong maximum principle} \label{sec:smp}

\noindent
In this short conclusive section, we show how the Hopf lemma of Theorem~\ref{Hopflemma} yields the following strong maximum (or, better, minimum) principle for the operator~$Q$.

\begin{proposition} \label{SMPprop}
Let~$\Omega \subset \R^n$ be a bounded open set. Assume that~$A$,~$B$, and~$\phi$ satisfy hypotheses~\eqref{ass:A},~\eqref{ass:B}, and~\eqref{ass:phi}. Let~$u \in W^{1, p}(\Omega) \cap \W^{s, q}(\Omega) \cap C^1(\Omega)$ be a non-negative weak supersolution of~$Qu = 0$ in~$\Omega$. Then, either~$u > 0$ in~$\Omega$ or~$u \equiv 0$ in~$\R^n$.
\end{proposition}

\begin{proof}
We already know from Proposition~\ref{WCPprop} that~$u \ge 0$ in~$\R^n$. Thus, to prove the proposition we suppose that there exists a point~$x_0 \in \Omega$ at which~$u(x_0) = 0$ and show that~$u$ must vanish identically in~$\R^n$. Let~$U$ be the connected component of~$\Omega$ containing~$x_0$. Note that~$U$ is a bounded connected open set. We begin by establishing that
\begin{equation} \label{u=0inOmega}
u = 0 \quad \mbox{in } U.
\end{equation}
The proof of this claim is standard. Nevertheless, we provide the details for the sake of completeness. Let~$U' \coloneqq \left\{ x \in U : u(x) = 0 \right\}$. Clearly,~$U'$ is relatively closed in~$U$ and non-empty, as~$u$ is continuous and~$x_0 \in U'$. Hence, its complement~$U \setminus U'$ is open. If it were non-empty, then there would exist a point~$x_1$ in it such that~$\dist(x_1, U') < \dist(x_1, \R^n \setminus U)$ and a radius~$r > 0$ for which~$B_r(x_1) \subset U \setminus U'$ and~$\partial B_r(x_1) \cap U' \ne \varnothing$. Applying Theorem~\ref{Hopflemma}, we would deduce that~$\frac{\partial u}{\partial \nu}(x_2) < 0$ at every point~$x_2 \in \partial B_r(x_1) \cap U'$. But this is a contradiction, since~$x_2$ is an interior minimum point for~$u$ and~$\nabla u(x_2)$ must therefore vanish,~$u$ being of class~$C^1$ inside~$U$. We conclude that~\eqref{u=0inOmega} holds true.

We now show that~$u$ must vanish outside of~$U$ as well. Here, the presence of the nonlocal operator~$Q_N$ plays a crucial role. In view of~\eqref{u=0inOmega} and of the fact that~$u$ is a weak supersolution, we infer that
$$
\iint_{\mathscr{C}_U} \phi \big( {u(x) - u(y)} \big) \left( \varphi(x) - \varphi(y) \right) \frac{B(x, y)}{|x - y|^{n + s q}} \, dx dy \ge 0,
$$
for every non-negative~$\varphi \in C^\infty_c(U)$.
Being~$\phi$ odd and~$B$ symmetric, by taking advantage of~\eqref{u=0inOmega} once again we find
$$
    \int_U \bigg( {\int_{\R^n\setminus U}\phi\big(u(y) \big)\frac{B(x,y)}{|x-y|^{n+sq}} \, dy} \bigg)\varphi(x)\,dx\leq 0 \quad \text{for every non-negative } \varphi\in C^\infty_c(U).
$$
Since~$u$ is non-negative,~$B$ is strictly positive in~$\R^n \times \R^n$, and~$\phi$ is strictly positive in~$(0, +\infty)$, thanks to assumptions~\eqref{ass:B} and~\eqref{ass:phi}, we deduce
that~$u = 0$ in~$\R^n \setminus U$, thus concluding the proof.
\end{proof}

\appendix
\section{Proof of Lemma~\ref{Linftyestlem}} \label{app:Linfty}

\noindent
We include here a proof of Lemma~\ref{Linftyestlem}, claiming global~$W^{1, p}$ and~$L^\infty$ bounds for the solutions of problem~\eqref{dirprobforu}. We begin by establishing the~$W^{1, p}$ estimate.

By testing the weak formulation of~\eqref{dirprobforu} with~$\varphi = u - g$, we get
\begin{align*}
& \int_\Omega A(x, Du) \cdot \big( {Du - Dg} \big) \, dx \\
& \hspace{15pt} + \iint_{\mathscr{C}_\Omega} \phi \big( {u(x) - u(y)} \big) \big( {u(x) - u(y) - g(x) + g(y)} \big) \frac{B(x, y)}{|x - y|^{n + s q}} \, dx dy = \int_\Omega f (u - g) \, dx.
\end{align*}
Taking advantage of assumption~\eqref{ass:A} and arguing similarly as in the proof of Lemma~\ref{lieblem}, we obtain the following estimate for the first summand on the left-hand side:
$$
\int_\Omega A(x, Du) \cdot \left( Du - Dg \right) \, dx \ge \frac{1}{C} \, \| Du \|_{L^p(\Omega)}^p - C \left( 1 + \| Dg \|_{L^p(\Omega)}^p \right).
$$
The second summand can be handled using hypothesis~\eqref{ass:B} and~\eqref{ass:phi} along with the weighted Young's inequality. We get
\begin{align*}
& \iint_{\mathscr{C}_\Omega} \phi \big( {u(x) - u(y)} \big) \big( {u(x) - u(y) - g(x) + g(y)} \big) \frac{B(x, y)}{|x - y|^{n + s q}} \, dx dy \\
& \hspace{40pt} \ge \frac{1}{C} \iint_{\mathscr{C}_\Omega} \frac{|u(x) - u(y)|^q}{|x - y|^{n + s q}} \, dx dy - C \iint_{\mathscr{C}_\Omega} \frac{|u(x) - u(y)|^{q - 1} |g(x) - g(y)|}{|x - y|^{n + s q}} \, dx dy \\
& \hspace{40pt} \ge \frac{1}{C} \iint_{\mathscr{C}_\Omega} \frac{|u(x) - u(y)|^q}{|x - y|^{n + s q}} \, dx dy - C \iint_{\mathscr{C}_\Omega} \frac{|g(x) - g(y)|^q}{|x - y|^{n + s q}} \, dx dy.
\end{align*}
Finally, thanks to the Sobolev (or Morrey) and Poincar\'e inequalities, the right-hand is estimated by
$$
\int_\Omega f (u - g) \, dx \le \| f \|_{L^n(\Omega)} \| u - g \|_{L^{\frac{n}{n - 1}}(\Omega)} \le C \| f \|_{L^n(\Omega)} \big( {\| Du \|_{L^p(\Omega)} + \| Dg \|_{L^p(\Omega)}} \big).
$$
Hence, using again the weighted Young's inequality, we find that
\begin{align*}
& \| Du \|_{L^p(\Omega)}^p + \iint_{\mathscr{C}_\Omega} \frac{|u(x) - u(y)|^q}{|x - y|^{n + s q}} \, dx dy \\
& \hspace{70pt} \le C \left( 1 + \| Dg \|_{L^p(\Omega)}^p + \iint_{\mathscr{C}_\Omega} \frac{|g(x) - g(y)|^q}{|x - y|^{n + s q}} \, dx dy + \| f \|_{L^n(\Omega)}^{\frac{p}{p - 1}} \right).
\end{align*}
The bound for~$\| u \|_{W^{1, p}(\Omega)}$ immediately follows from this and Poincar\'e's inequality.

We now deal with the global boundedness of~$u$ in~$\Omega$, which we establish it via the De Giorgi-Stampacchia method. Clearly, in the supercritical case~$p > n$ the bound is a consequence of Morrey's inequality and the previous estimate for~$\| u \|_{W^{1, p}(\Omega)}$. Assume then that~$n \ge p$.

Let~$k > M > M_1$, with
$$
M_1 \coloneqq 1 + \| g \|_{L^\infty(\Omega')} + \left( \int_{\R^n} \frac{|g(y)|^{q - 1}}{\left( 1 + |y| \right)^{n + s q}} \, dy \right)^{\! \frac{1}{p - 1}} + \| f \|_{L^n(\Omega)}^{\frac{1}{p - 1}}.
$$
Observe that the quantity on the right-hand side is finite, thanks to the inequality
\begin{equation} \label{tailest}
\int_{\R^n} \frac{|g(y)|^{q - 1}}{\left( 1 + |y| \right)^{n + s q}} \, dy \le C \left( \| g \|_{L^q(\Omega')} + \iint_{\mathscr{C}_\Omega} \frac{|g(x) - g(y)|^q}{|x - y|^{n + s q}} \, dx dy  \right)^{\! q - 1},
\end{equation}
which is easily established by using~\eqref{bounds:omega}. As~$k > \| g \|_{L^\infty(\Omega')}$, the function~$\varphi \coloneqq (u - k)_+ \chi_{\Omega}$ lies in~$W^{1, p}_0(\Omega) \cap \W^{s, q}_0(\Omega)$, and can thus be plugged in the weak formulation of problem~\eqref{dirprobforu}. Setting~$\Omega_k \coloneqq \left\{ x \in \Omega : u(x) > k \right\}$, we obtain
\begin{equation} \label{eqagainstu-k}
\begin{aligned}
& \int_{\Omega_k} f (u - k) \, dx = \int_{\Omega_k} A(x, Du) \cdot Du \, dx \\
& \hspace{30pt} + \iint_{\mathscr{C}_\Omega} \frac{\phi \big( {u(x) - u(y)} \big) \big( {(u(x) - k)_+ \chi_{\Omega}(x) - (u(y) - k)_+ \chi_{\Omega}(y)} \big) B(x, y)}{|x - y|^{n + s q}} \, dx dy.
\end{aligned}
\end{equation}
On the one hand, we clearly have that
\begin{equation} \label{locallb}
\begin{aligned}
\int_{\Omega_k} A(x, Du) \cdot Du \, dx & \ge \frac{1}{C} \int_{\Omega_k} \left( |Du|^2 + \mu^2 \right)^{\frac{p - 2}{2}} |Du|^2 \, dx \\
& \ge \frac{1}{C} \Big( {\big[ {(u - k)_+} \big]_{W^{1, p}(\Omega)}^p - \mu^p |\Omega_k|} \Big).
\end{aligned}
\end{equation}
Regarding the nonlocal term, thanks to the oddness of~$\phi$ we observe that
\begin{align*}
& \phi \big( {u(x) - u(y)} \big) \big( {(u(x) - k)_+ \chi_{\Omega}(x) - (u(y) - k)_+ \chi_{\Omega}(y)} \big) \\
& \hspace{70pt} = \begin{cases}
\phi \big( {u(x) - u(y)} \big) \big( {u(x) - u(y)} \big) & \quad \mbox{if } x, y \in \Omega_k, \\
\phi \big( {u(x) - u(y)} \big) \big( {u(x) - k} \big) & \quad \mbox{if } x \in \Omega_k, \, y \in \R^n \setminus \Omega_k, \\
\phi \big( {u(y) - u(x)} \big) \big( {u(y) - k} \big) & \quad \mbox{if } x \in \R^n \setminus \Omega_k, \, y \in \Omega_k, \\
0 & \quad \mbox{if } x, y \in \R^n \setminus \Omega_k.
\end{cases}
\end{align*}
Recalling~\eqref{ass:B}-\eqref{ass:phi} and using that~$u(y) = g(y) \le k$ for a.e.~$y \in \Omega' \setminus \Omega$, we conclude from the previous identity that
\begin{equation} \label{nonlocallb}
\begin{aligned}
&\iint_{\mathscr{C}_\Omega} \frac{\phi \big( {u(x) - u(y)} \big) \big( {(u(x) - k)_+ \chi_{\Omega}(x) - (u(y) - k)_+ \chi_{\Omega}(y)} \big) B(x, y)}{|x - y|^{n + s q}} \, dx dy \\
& \hspace{130pt} \ge \frac{1}{C} \int_{\Omega'} \int_{\Omega'} \frac{\big| {\big( {u(x) - k} \big)_+ - \big( {u(y) - k} \big)_+} \big|^q}{|x - y|^{n + sq}} \, dx dy \\
& \hspace{130pt} \quad - C \int_{\Omega_k} \left( \int_{\R^n \setminus \Omega'} \frac{(u(x) - k) \big( {g(y) - u(x)} \big)_+^{q - 1}}{|x - y|^{n + sq}} \, dy \right) dx \\
& \hspace{130pt} \ge - C \, k^{p - 1} \, \| (u - k)_+ \|_{L^1(\Omega)},
\end{aligned}
\end{equation}
where we used that
\begin{align*}
\int_{\Omega_k} \left( \int_{\R^n \setminus \Omega'} \!\! \frac{(u(x) - k) \big( {g(y) - u(x)} \big)_+^{q - 1}}{|x - y|^{n + sq}} \, dy \right) dx & \le \int_{\Omega_k} (u(x) - k) \left( \int_{\R^n \setminus \Omega'} \! \frac{|g(y)|^{q - 1}}{|x - y|^{n + s q}} \, dy \right) dx \\
& \le C \| (u - k)_+ \|_{L^1(\Omega)} \int_{\R^n} \frac{|g(y)|^{q - 1}}{\left( 1 + |y| \right)^{n + s q}} \, dy.
\end{align*}
Finally, we estimate
$$
\int_{\Omega_k} f (u - k) \, dx \le \| f \|_{L^n(\Omega)} \| (u - k)_+ \|_{L^{\frac{n}{n - 1}}(\Omega)} \le C \, k^{p - 1} \, \| (u - k)_+ \|_{L^{\frac{n}{n - 1}}(\Omega)}.
$$
Combining this with~\eqref{eqagainstu-k},~\eqref{locallb},~\eqref{nonlocallb}, and the Poincar\'e-Sobolev inequality in~$W^{1,p}_0(\Omega)$, we get that
$$
\| (u - k)_+ \|_{L^{m p}(\Omega)}^p \le C \Big( {k^p |\Omega_k| + k^{p - 1} \| (u - k)_+ \|_{L^1(\Omega)}} + k^{p - 1} \| (u - k)_+ \|_{L^\frac{n}{n - 1}(\Omega)} \Big),
$$
where~$m$ is equal to~$\frac{n}{n - p}$ when~$n > p$ or to any number strictly larger than~$\frac{n}{n- 1}$ when~$n = p$. From this and the inequalities
\begin{align*}
\| (u - k)_+ \|_{L^{\frac{n}{n - 1}}(\Omega)} & \le (k - h)^{1 - \frac{(n - 1)m p}{n}} \| (u - h)_+ \|_{L^{m p}(\Omega)}^{\! \frac{(n - 1)m p}{n}}, \\
|\Omega_k| & \le |\Omega|^{\frac{1}{n}} (k - h)^{- \frac{(n - 1) m p}{n}} \| (u - h)_+ \|_{L^{m p}(\Omega)}^{\! \frac{(n - 1)m p}{n}}, \\
\| (u - k)_+ \|_{L^1(\Omega)} & \le |\Omega|^{\frac{1}{n}} (k - h)^{1 - \frac{(n - 1)m p}{n}} \| (u - h)_+ \|_{L^{m p}(\Omega)}^{\! \frac{(n - 1) m p}{n}},
\end{align*}
valid for any~$h \in (0, k)$, we find
$$
\| (u - k)_+ \|_{L^{m p}(\Omega)}^p \le C \, \frac{k^p}{ (k - h)^{\frac{(n - 1)m p}{n}} } \, \| (u - h)_+ \|_{L^{m p}(\Omega)}^{\! \frac{(n - 1)m p}{n}}.
$$

Letting now~$\delta \coloneqq \frac{(n - 1) m - n}{n} > 0$,~$k_i \coloneqq (2 - 2^{-i}) M$, and~$\Psi_i \coloneqq \| (u - k_i)_+ \|_{L^{m p}(\Omega)}^p$ for every~$i \in \N \cup \{ 0 \}$, we infer that
$$
\Psi_{i + 1} \le \frac{C}{M^{\delta p}} \, 2^{(1 + \delta) p i} \, \Psi_i^{1 + \delta} \quad \mbox{for every } i \in \N \cup \{ 0 \}.
$$
By taking advantage of~\cite[Lemma~7.1]{G03}, we conclude that~$\displaystyle \| (u - 2M)_+ \|_{L^{m p}(\Omega)}^p = \lim_{i \rightarrow +\infty} \Psi_i = 0$, i.e.,~$u \le 2 M$ in~$\Omega$, provided~$\| (u - M)_+ \|_{L^{m p}(\Omega)}^p = \Psi_0 \le C^{-1} M^p$ for some constant~$C \ge 1$ large enough. Clearly, this can be achieved by taking~$M \coloneqq M_1 + M_2$ with
$$
M_2 \coloneqq C \| u \|_{W^{1, p}(\Omega)}
$$
and~$C \ge 1$ sufficiently large, thanks to the Sobolev inequality.

We thus established an upper bound for~$u$. Since a corresponding lower bound can be obtained analogously, we conclude that the proof of Lemma~\ref{Linftyestlem} is complete---also recall the tail estimate~\eqref{tailest}.

\section{Proof of Lemma~\ref{lemma:ext}} \label{app:Hextproof}

\noindent
This appendix is devoted to the proof of the extension Lemma~\ref{lemma:ext}, used within Section~\ref{hopfsec}. We begin by constructing a~$C^{1, \alpha}(\overline{\R^n_+}) \cap C^\infty(\R^n_+)$-extension~$H$ of~$h$ satisfying~$DH(0) = 0$,~\eqref{stime:H1}, and
\begin{equation} \label{stime:H}
\| H \|_{ C^{1, \a} \left( \overline{\R^n_+} \right)} \le C \, \| h \|_{C^{1, \a}(\R^{n-1})}.
\end{equation}
We take as~$H$ the harmonic extension of~$h$ to~$\R^n_+$, namely the unique bounded solution of
    \begin{equation}\label{eq:H}
        \begin{cases}
            \Delta H = 0\quad & \text{in } \R^n_+,
            \\
            H=h \quad & \text{on } \partial \R^n_+.
        \end{cases}
    \end{equation}
Standard regularity theory yields that~$H$ is of class~$C^{1,\a}_{\loc}(\overline{\R^n_+}) \cap C^\infty(\R^n_+)$. Furthermore, since~$\partial_{y_n} H(\cdot,0)=- (-\Delta)^{1/2} h$, from~\eqref{h00} and~\eqref{halflapofhis0} we infer that~$DH(0)=0$.

We now address the proof of estimates~\eqref{stime:H} and~\eqref{stime:H1}, which will both follow from the Poisson representation formula
\begin{equation}\label{repres}
    \begin{aligned}
        H(y', y_n) &  = \frac{2 y_n}{n \omega_n} \int_{\R^{n - 1}} \frac{h(z')}{\left( |y' - z'|^2 + y_n^2 \right)^{\frac{n}{2}}} \, dz'
        \\
        & =\frac{2 y_n}{n \omega_n} \int_{\R^{n - 1}} \frac{h(y' + \ell')}{\left( |\ell'|^2 + y_n^2 \right)^{\frac{n}{2}}} \, d\ell'
        \\
        &= \frac{2}{n \omega_n} \,\int_{\R^{n-1}}\frac{h(y'+y_n \tau')}{(1+|\tau'|^2)^{\frac{n}{2}}}\,d\tau',
    \end{aligned}
\end{equation}
valid for~$(y',y_n)\in \R^n_+$. Here we set~$\omega_n=|B_1|$. From~\eqref{repres} and the fact that
\begin{equation} \label{1=1}
\frac{2 y_n}{n \omega_n} \,\int_{\R^{n-1}}\frac{d\ell'}{(|\ell'|^2 + y_n^2)^{\frac{n}{2}}} = 1 \quad \mbox{for every } y_n > 0,
\end{equation}
we immediately infer that
\begin{equation} \label{Hleh}
\| H \|_{L^\infty(\R^n_+)} \le \| h \|_{L^\infty(\R^{n - 1})}.
\end{equation}

By differentiating the second identity in~\eqref{repres}, we get
$$
D'H(y',y_n) = \frac{2 y_n}{n \omega_n} \int_{\R^{n-1}}\frac{D' h(y' + \ell')}{\left( |\ell'|^2 + y_n^2 \right)^{\frac{n}{2}}} \, d\ell'.
$$
From this and~\eqref{1=1}, it readily follows that
\begin{equation}\label{D'H}
\big| {D'H(y',y_n)} \big| \le \| D'h \|_{L^\infty(\R^{n-1})} \quad \mbox{for all } y' \in \R^{n - 1}, \, y_n > 0
\end{equation}
and
\begin{equation}\label{D'Ha}
\big {|D'H(y',y_n)-D'H(z',y_n)} \big| \le [D'h]_{C^\alpha(\R^{n - 1})} \, |y' - z'|^\a \quad \mbox{for all } y', z' \in \R^{n - 1}, \, y_n > 0.
\end{equation}
Through a suitable change of variables, we then compute
\begin{equation}\label{dnH}
\begin{aligned}
\partial_{y_n} H(y',y_n) & = \frac{2}{n \omega_n} \int_{\R^{n-1}}\frac{|\ell'|^2-(n-1) y_n^2}{\left(|\ell'|^2+y_n^2\right)^{\frac{n+2}{2}}}\,h(y'+\ell')\,d\ell' \\
& = \frac{2}{n \omega_n y_n}\int_{\R^{n-1}}\frac{|\tau'|^2-(n-1)}{\left(1 + |\tau'|^2 \right)^{\frac{n+2}{2}}}\,h(y'+y_n \tau')\,d\tau' \\
& =\frac{2}{n \omega_n y_n} \Bigg\{ {\int_{\R^{n-1}\setminus B'_M}\frac{|\tau'|^2-(n-1)}{(1+|\tau'|^2)^{\frac{n+2}{2}}}\, \Big( {h(y'+y_n \tau') - h(y')} \Big) \, d\tau'} \\
& \quad + {\int_{B'_M}\frac{|\tau'|^2-(n-1)}{(1+|\tau'|^2)^{\frac{n+2}{2}}}\Big( {h(y'+ y_n \tau')-h(y')-D'h(y')\cdot y_n \tau'} \Big) \, d\tau'} \Bigg\},
\end{aligned}
\end{equation}
for~$M > 0$ to be chosen. Note that for the last identity we took advantage of the identities
\begin{equation}\label{identities}
\int_{\R^{n-1}}\frac{|\tau'|^2-(n-1)}{(1+|\tau'|^2)^{\frac{n+2}{2}}}\,d\tau'=0\quad\text{and}\quad \int_{B'_M}\frac{|\tau'|^2-(n-1)}{(1+|\tau'|^2)^{\frac{n+2}{2}}}\,w'\cdot\tau'\,d\tau'=0,
\end{equation}
valid for every~$w'\in \R^{n-1}$ and~$M>0$. By taking~$M = y_n^{-1}$ in~\eqref{dnH}, we deduce that
\begin{equation}\label{dnH1}
\begin{aligned}
\big| {\partial_{y_n}H(y',y_n)} \big| & \leq \frac{C}{y_n}\,\Bigg\{ { \| h \|_{L^\infty(\R^{n - 1})} \int_M^\infty\frac{d\rho}{\rho^2} + [D'h]_{C^\alpha(\R^{n - 1})} y_n^{1+\a} \int_0^M \frac{\rho^{n - 1 + \alpha}}{(1+\rho)^{n}} \, d\rho} \Bigg\} \\
& \leq C\, \Big\{ {\| h \|_{L^\infty(\R^{n - 1})} (M y_n)^{-1} + [D'h]_{C^\alpha(\R^{n - 1})} (M y_n)^\alpha } \Big\} \\
& \leq C \, \| h \|_{C^{1,\a}(\R^{n-1})} \quad \mbox{for all } y' \in \R^{n - 1}, \, y_n > 0.
\end{aligned}
\end{equation}
Arguing as for~\eqref{dnH}, we also have that
\begin{align*}
& \big| {\partial_{y_n}H(y',y_n) - \partial_{y_n} H(z', y_n)} \big| \\
& \hspace{10pt} \le \frac{C}{y_n} \Bigg\{ {\int_{\R^{n-1}\setminus B'_N} \! \frac{\big| {h(y'+y_n \tau') - h(y') - h(z' + y_n \tau') + h(z')} \big|}{(1+|\tau'|)^n} \, d\tau'} \\
& \hspace{10pt} \quad + {\int_{B'_N} \! \frac{\big| {h(y'+ y_n \tau') - h(y') - h(z' + y_n \tau') + h(z') - \big( {D'h(y') - D'h(z')} \big) \cdot y_n \tau'} \big| }{(1+|\tau'|)^n} \, d\tau'} \Bigg\},
\end{align*}
for any radius~$N > 0$. Using the fundamental theorem of calculus, it is not difficult to see that the numerator of the fraction inside the first integral is bounded by~$\| D'h \|_{C^{\alpha}(\R^{n - 1})} |y' - z'| \min \big\{ {2, y_n^\alpha |\tau'|^\alpha} \big\}$, while that pertaining to the second integral by~$2 [D'h]_{C^\alpha(\R^{n - 1})} y_n^{1 + \alpha} |\tau'|^{1 + \alpha}$. In light of these estimates, choosing~$N = |y' - z'| y_n^{-1}$ we get that
\begin{align*}
& \big| {\partial_{y_n}H(y',y_n) - \partial_{y_n} H(z', y_n)} \big| \\
& \hspace{5pt} \le \frac{C \, \| D'h \|_{C^{\alpha}(\R^{n - 1})}}{y_n} \left\{ |y' - z'| \int_{y_n^{-1}}^{+\infty} \frac{d\rho}{\rho^2} + y_n^\alpha |y' - z'| \int_{N}^{y_n^{-1}} \frac{d\rho}{\rho^{2 - \alpha}} + y_n^{1 + \alpha} \int_0^{N} \frac{\rho^{n - 1 + \alpha}}{(1 + \rho)^n} \, d\rho \right\} \\
& \hspace{5pt} \le C \, \| D'h \|_{C^{\alpha}(\R^{n - 1})} \left\{ |y' - z'| + \frac{|y' - z'|}{(N y_n)^{1 - \alpha}} + (N y_n)^{\alpha} \right\} \le C \, \| D'h \|_{C^\alpha(\R^{n - 1})} |y' - z'|^\alpha,
\end{align*}
for every~$y', z' \in \R^{n - 1}$ such that~$|y' - z'| < 1$ and for every~$y_n > 0$. Combining this with~\eqref{dnH1}, we conclude that
\begin{equation} \label{dnHa}
\big| {\partial_{y_n}H(y',y_n) - \partial_{y_n} H(z', y_n)} \big| \le C \| h \|_{C^{1, \alpha}(\R^{n - 1})} |y' - z'|^\alpha \quad \mbox{for } y', z' \in \R^{n - 1}, \, y_n > 0.
\end{equation}
   
Now, differentiating the first identity in \eqref{repres}, we obtain the following alternative expression for the horizontal gradient of~$H$:
$$
D'H(y',y_n) = \frac{2 y_n}{\omega_n} \int_{\R^{n-1}}\frac{h(z')\,(z'-y')}{( |y'-z'|^2+y_n^2)^{\frac{n+2}{2}}}\,dz' = \frac{2}{\omega_n y_n}\int_{\R^{n-1}}\frac{h(y'+y_n \tau') \, \tau'}{(1+|\tau'|^2)^{\frac{n+2}{2}}} \, d\tau'.
$$
Therefore, for all~$i,j=1,\dots,n-1$ we have, by symmetry,
\begin{align*}
\partial^2_{y_i y_j} H(y',y_n) & =\frac{2}{\omega_n y_n}\int_{\R^{n-1}}\frac{\partial_{y'_i}h(y'+y_n \tau') \, \tau_j}{(1+|\tau'|^2)^{\frac{n+2}{2}}}\,d\tau' \\
& = \frac{2}{\omega_n y_n} \int_{\R^{n-1}} \frac{\Big( {\partial_{y'_i}h(y'+y_n \tau')-\partial_{y'_i} h(y')} \Big) \tau_j}{(1+|\tau'|^2)^{\frac{n+2}{2}}} \,d\tau',
\end{align*}
so that
\begin{equation} \label{dijH'}
\begin{aligned}
\big| { \partial^2_{y_i y_j} H(y',y_n)} \big| & \leq C \, [D'h]_{C^\alpha(\R^{n - 1})} \, y_n^{\alpha - 1} \int_{\R^{n-1}} \frac{|\tau'|^{1+\a}}{(1+|\tau'|)^{n+2}} \, d\tau' \\
& \leq C \, [D'h]_{C^\alpha(\R^{n - 1})} \, y_n^{\alpha - 1}.
\end{aligned}
\end{equation}
Next, from the second identity in \eqref{dnH}, we get
\begin{align*}
\partial^2_{y_i y_n} H(y',y_n) & = \frac{2}{n \omega_n y_n} \int_{\R^{n-1}}\frac{|\tau'|^2-(n-1)}{\big( 1+|\tau'|^2 \big)^{\frac{n+2}{2}}}\,\partial_{y'_i} h(y'+y_n \tau')\,d\tau' \\
& = \frac{2}{n \omega_n y_n} \int_{\R^{n-1}} \frac{|\tau'|^2-(n-1)}{\big( 1+|\tau'|^2 \big)^{\frac{n+2}{2}}} \big( {\partial_{y'_i} h(y'+y_n \tau) - \partial_{y_i'} h(y')} \big) \, d\tau',
\end{align*}
where we also made use of the first identity in~\eqref{identities}. Therefore, we estimate
\begin{equation}\label{dinH'}
\big| {\partial^2_{y_i y_n} H(y',y_n)} \big| \leq C \, [D'h]_{C^\a(\R^{n - 1})} \, y_n^{\alpha - 1} \int_{\R^{n-1}} \frac{|\tau'|^\a}{(1+|\tau'|)^n} \, d\tau' \le C \, [D'h]_{C^\a(\R^{n - 1})} \, y_n^{\a-1}.
\end{equation}
Since from the equation in~\eqref{eq:H} we know that~$\partial^2_{y_n y_n} H = - \sum_{i = 1}^{n - 1} \partial^2_{y_i y_i} H$ in~$\R^n_+$, from~\eqref{dijH'} it also follows that~$\big| {\partial^2_{y_n y_n} H(y',y_n)} \big| \le C \, [D'h]_{C^\a(\R^{n - 1})} \, y_n^{\a-1}$. By combining this with~\eqref{dijH'} and~\eqref{dinH'}, we conclude that~\eqref{stime:H1} holds true.

Finally, from~\eqref{stime:H1} and the fundamental theorem of calculus, we find that
$$
\big| {DH(y',y_n)-DH(y',z_n)} \big| = \left| \int_{z_n}^{y_n} \partial_{y_n} DH (y',t)\,dt \right| \leq C\,[D'h]_{C^\a(\R^{n-1})}\,|y_n-z_n|^\a.
$$
By putting together this with~\eqref{Hleh},~\eqref{D'H},~\eqref{D'Ha},~\eqref{dnH1}, and~\eqref{dnHa}, we obtain~\eqref{stime:H}.

To conclude the proof, it suffices to consider any~$C^{1, \alpha}(\R^n)$-extension of~$H$ to the whole~$\R^n$ having~$C^{1, \alpha}(\R^n)$ norm bounded by that of~$H$, up to a factor. This can be done, for instance, via the elegant approach of~\cite{S64}--see also~\cite[Theorem 1.1.17]{maz}.

\section*{Acknowledgments}
\noindent
The authors have been partially supported by the ``Gruppo Nazionale per l'Analisi Matematica, la Probabilit\`a e le loro Applicazioni'' (GNAMPA) of the ``Istituto Nazionale di Alta Matematica'' (INdAM, Italy).

%
%
%
%
%


\end{document}